\documentclass{amsart}
\usepackage{amssymb}
\usepackage{xcolor}
\usepackage{bbm}

\theoremstyle{plain}
\newtheorem{theorem}{Theorem}
\newtheorem{lemma}[theorem]{Lemma}

\theoremstyle{definition}

\theoremstyle{remark}
\newtheorem{remark}[theorem]{Remark}
\newtheorem{example}[theorem]{Example}

\providecommand{\abs}[1]{\lvert#1\rvert}
\providecommand{\Abs}[1]{\Bigl\lvert#1\Bigr\rvert}
\providecommand{\norm}[1]{\lVert#1\rVert}

\usepackage{graphicx}
\usepackage{wrapfig,lipsum}
\usepackage{multirow}
\usepackage{subcaption}
\usepackage{afterpage}
\usepackage{float}
\usepackage{multicol}
\usepackage{algorithm2e}
\usepackage[colorlinks=true, linkcolor=blue, citecolor=blue, urlcolor=blue]{hyperref}
\usepackage{array}
\usepackage{caption}

\newcommand{\bs}{\boldsymbol}

\begin{document}

\title[Asymptotics of predictive distributions]{Asymptotics of predictive distributions driven by sample means and variances}
\author{Samuele Garelli}
\address{Samuele Garelli, Dipartimento di Scienze Statistiche ``P. Fortunati'', Universit\`a di Bologna, via delle Belle Arti 41, 40126 Bologna, Italy}
\email{samuele.garelli2@unibo.it}
\author{Fabrizio Leisen}
\address{Fabrizio Leisen, Department of Mathematics, King’s College, Strand WC2R 2LS, London, UK}
\email{fabrizio.leisen@gmail.com}
\author{Luca Pratelli}
\address{Luca Pratelli, Accademia Navale, viale Italia 72, 57100 Livorno,
Italy} \email{luca{\_}pratelli@marina.difesa.it}
\author{Pietro Rigo}
\address{Pietro Rigo (corresponding author), Dipartimento di Scienze Statistiche ``P. Fortunati'', Universit\`a di Bologna, via delle Belle Arti 41, 40126 Bologna, Italy}
\email{pietro.rigo@unibo.it}

\begin{abstract}
Let $\alpha_n(\cdot)=P\bigl(X_{n+1}\in\cdot\mid X_1,\ldots,X_n\bigr)$ be the predictive distributions of a sequence $(X_1,X_2,\ldots)$ of $p$-dimensional random vectors. Suppose
$$\alpha_n=\mathcal{N}(M_n,Q_n)$$
where $M_n=\frac{1}{n}\sum_{i=1}^nX_i$ and $Q_n=\frac{1}{n}\sum_{i=1}^n(X_i-M_n)(X_i-M_n)^t$. Then, there is a random probability measure $\alpha$ on the Borel subsets of $\mathbb{R}^p$ such that $\norm{\alpha_n-\alpha}\overset{a.s.}\longrightarrow 0$ where $\norm{\cdot}$ is total variation distance. An explicit expression for $\alpha$ is provided and the convergence rate of $\norm{\alpha_n-\alpha}$ is shown to be arbitrarily close to $n^{-1/2}$. Moreover, it is still true that $\norm{\alpha_n-\alpha}\overset{a.s.}\longrightarrow 0$ even if $\alpha_n=\mathcal{L}(M_n,Q_n)$ where $\mathcal{L}$ belongs to a class of distributions much larger than the normal. The predictives $\alpha_n$ are useful in various frameworks, including Bayesian predictive inference and predictive resampling. Finally, the asymptotic behavior of copula-based predictive distributions (introduced in \cite{HMW18}) is investigated and a numerical experiment is performed.
\end{abstract}

\maketitle

\section{introduction}\label{intro}

\noindent Let $X=(X_1,X_2,\ldots)$ be a sequence of random variables, with values in a standard Borel space $(S,\mathcal{B})$, and
$$X(n)=(X_1,\ldots,X_n).$$
The {\em predictive distributions} of $X$ are
$$\alpha_0(\cdot)=P(X_1\in\cdot)\quad\text{and}\quad \alpha_n(\cdot)=P\bigl[X_{n+1}\in\cdot\mid X(n)\bigr]\quad\text{for each }n\ge 1.$$
Thus, $\alpha_0$ is the marginal distribution of $X_1$ and $\alpha_n$ the conditional distribution of $X_{n+1}$ given $X(n)$. Note that $\alpha_0$ is a fixed probability measure on $\mathcal{B}$ while $\alpha_n$ is a random probability measure on $\mathcal{B}$ for $n\ge 1$.

\medskip

\noindent By the Ionescu-Tulcea theorem, in order to assess the probability distribution of $X$, it suffices to specify the sequence $(\alpha_n)$; see e.g. \cite{STS23}. Precisely, suppose the $X_n$ are defined on a set $\Omega$ and $X(\Omega)=S^\infty$. (The condition $X(\Omega)=S^\infty$ is just to avoid unnecessary complications; it is automatically true if $\Omega=S^\infty$ and the $X_n$ are the canonical projections). Select any sequence $(\alpha_n:n\ge 0)$ such that:

\medskip

\begin{itemize}

\item[(i)] $\alpha_0\in\mathcal{P}$ and $\alpha_n:\Omega\rightarrow\mathcal{P}$ for $n\ge 1$, where $\mathcal{P}$ denotes the collection of all probability measures on $\mathcal{B}$;

\item[(ii)] The map $\omega\in\Omega\mapsto\alpha_n(\omega,A)$ is measurable with respect to $\sigma\bigl[X(n)\bigr]$ for fixed $n\ge 1$ and $A\in\mathcal{B}$.

\end{itemize}

\medskip

\noindent Then, there is a unique probability measure $P$ on $\sigma(X)$ which admits the $\alpha_n$ as predictive distributions. Importantly, apart from conditions (i)-(ii), the $\alpha_n$ are {\em arbitrary}. Roughly speaking, to assign the distribution of $X$, it suffices to select (arbitrarily) the marginal distribution of $X_1$, the conditional distribution of $X_2$ given $X_1$, the conditional distribution of $X_3$ given $(X_1,X_2)$, and so on. This fact has various consequences.

\medskip

\noindent For instance, consider a Bayesian forecaster who aims to predict $X_{n+1}$ based on $X(n)$ for every $n$. To this end, it is (necessary and) sufficient to select $(\alpha_n)$. This suggests a few remarks.

\begin{itemize}

\item To make Bayesian predictions, the data sequence $X$ is not forced to have any specific distributional form (such as exchangeable, stationary, Markov, and so on).

\item However, in the special case where $X$ is required to be exchangeable, it is possible to make predictions without explicitly selecting a prior. In fact, in the exchangeable case, selecting a prior or selecting the predictives $(\alpha_n)$ are two equivalent strategies to determine the distribution of $X$. Sometimes, the second strategy (i.e. assigning $(\alpha_n)$ directly) is more convenient.

\item By choosing $(\alpha_n)$, the forecaster is attaching probabilities to observable facts only. The value of $X_{n+1}$ is actually
observable, while the prior (being a probability on some random parameter) does not necessarily deal with observable facts.

\item In some Bayesian problems, even if the primary goal is not prediction, it may be convenient to assign the distribution of $X$ through the choice of $(\alpha_n)$. This happens mainly in Bayesian nonparametrics.

\end{itemize}

\medskip

\noindent In a nutshell, these are the basic ideas underlying the predictive approach to Bayesian inference; see \cite{BERN21}, \cite{STATSIN23}, \cite{STS23} and references therein.

\medskip

\noindent Finally, in addition to Bayesian predictive inference, there are other frameworks where the choice of $(\alpha_n)$ is the main task. Without any claim of being exhaustive, we mention machine learning \cite{CFZ}, \cite{DSGS}, \cite{EFRON}, \cite{HT}, causal inference \cite{LDM}, species sampling sequences \cite{PIT1996}, \cite{PITYOR}, Dawid's prequential approach \cite{DAW}, \cite{DAWVOVK}, and martingale posterior distributions \cite{FHW23}, \cite{HMW18}.

\subsection{Convergence of predictive distributions}\label{x4rg7h}

After selecting the sequence $(\alpha_n)$ of predictives, a natural question is whether it converges, in some sense, as $n\rightarrow\infty$. To formalize this question, suppose $X$ is defined on the probability space $(\Omega,\sigma(X),P)$ and $\alpha$ is a random probability measure (r.p.m.) on $\mathcal{B}$ defined on such a probability space. Recall also that the {\em total variation distance} between two probability measures on $\mathcal{B}$, say $\mu$ and $\nu$, is
$$\norm{\mu-\nu}=\sup_{A\in\mathcal{B}}\,\abs{\mu(A)-\nu(A)}.$$

\medskip

\noindent Then, $\alpha_n${\em converges weakly to $\alpha$ a.s.} if
\begin{gather*}
\alpha_n(\omega,\cdot)\overset{weakly}\longrightarrow\alpha(\omega,\cdot),\quad\text{as }n\rightarrow\infty,\text{ for }P\text{-almost all }\omega\in\Omega.
\end{gather*}
Similarly, $\alpha_n$ {\em converges in total variation to $\alpha$ a.s.} if
\begin{gather*}
\norm{\alpha_n(\omega,\cdot)-\alpha(\omega,\cdot)}\longrightarrow 0,\quad\text{as }n\rightarrow\infty,\text{ for }P\text{-almost all }\omega\in\Omega.
\end{gather*}
Clearly, $\alpha_n\rightarrow\alpha$ in total variation a.s. implies  $\alpha_n\rightarrow\alpha$ weakly a.s. but not conversely.

\medskip

\noindent Usually, convergence of $\alpha_n$ is a desirable property. This claim can be supported by various examples and remarks. Here, we report three of them.

\medskip

\begin{example}
Suppose $\alpha_n\rightarrow\alpha$ weakly a.s. for some r.p.m. $\alpha$ on $\mathcal{B}$. Then, $\mu_n\rightarrow\alpha$ weakly a.s. where $\mu_n=\frac{1}{n}\,\sum_{i=1}^n\delta_{X_i}$ is the empirical measure. Moreover, $X$ is asymptotically exchangeable, in the sense that $(X_n,X_{n+1},\ldots)\overset{dist}\longrightarrow Z$, as $n\rightarrow\infty$, for some exchangeable sequence $Z=(Z_1,Z_2,\ldots)$. As an aside, exploiting convergence of $(\alpha_n)$, one obtains the following characterization of exchangeability
$$X\text{ exchangeable}\quad\Longleftrightarrow\quad X\text{ stationary and }(\alpha_n)\text{ converges weakly a.s.}$$
In fact, ``$\Rightarrow$" is very well known; see e.g. \cite{AOP13}. As to ``$\Leftarrow$", just note that stationarity of $X$ implies $X\sim (X_n,X_{n+1},\ldots)$ for each $n$. Since $(\alpha_n)$ converges weakly a.s., one also obtains $(X_n,X_{n+1},\ldots)\overset{dist}\longrightarrow Z$ for some exchangeable $Z$. Hence, $X\sim Z$.
\end{example}

\medskip

\begin{example}\label{u8j4z2q1} Fix a $\sigma$-finite measure $\lambda$ on $\mathcal{B}$ and suppose that $X(n)$ has a density with respect to $\lambda^n$ for each $n$. Suppose also that $X$ is exchangeable or, more generally, {\em conditionally identically distributed} (in the sense of \cite{AOP04}). In this case, there is a r.p.m. $\alpha$ on $\mathcal{B}$ such that $\alpha_n\rightarrow\alpha$ weakly a.s. and an obvious question is whether $\alpha$ still admits a density with respect to $\lambda$. Precisely, this means that
\begin{gather}\label{crf}
\alpha(\omega,dx)=f(\omega,x)\,\lambda(dx)
\end{gather}
for some non-negative measurable function $f$ and $P$-almost all $\omega\in\Omega$. By Theorem 1 of \cite{AOP13}, the answer is
$$\text{Condition \eqref{crf} holds}\quad\Longleftrightarrow\quad \alpha_n\rightarrow\alpha\text{ in total variation a.s.}$$
\end{example}

\medskip

\begin{example}\label{g7m8uj} A new method for making Bayesian inference, referred to as {\em predictive resampling}, has been introduced by Fong, Holmes and Walker in \cite{FHW23}. A quick description of this method is provided in Section \ref{g7uj1}. Here, we just note that a fundamental step of predictive resampling is {\em the choice of a converging sequence $(\alpha_n)$ of predictive distributions}. Indeed, in a specific problem, predictive resampling is more or less effective depending on whether the choice of $(\alpha_n)$ is more or less suitable to the available data. In \cite{FHW23}, convergence of $(\alpha_n)$ is proved via martingale arguments and the corresponding posteriors are called martingale posteriors. However, martingales are not mandatory. In order to apply predictive resampling, it is enough to select $(\alpha_n)$ and to show that it converges (in some sense).
\end{example}

\subsection{Our contribution}\label{x6yh9}

\noindent This paper introduces and investigates a new sequence $(\alpha_n)$ of predictive distributions for the case
$$S=\mathbb{R}^p.$$
The $\alpha_n$ are specified in Section \ref{r4es3}. Here, we just note that each $\alpha_n$ depends on the data $X(n)$ only through the sample mean and covariance matrix. A meaningful example is
$$\alpha_n=\mathcal{N}(M_n,Q_n)$$
where $M_n=\frac{1}{n}\,\sum_{i=1}^nX_i$, $Q_n=\frac{1}{n}\,\sum_{i=1}^n(X_i-M_n)(X_i-M_n)^t$ and $\mathcal{N}(a,B)$ denotes the $p$-dimensional Gaussian distribution with mean $a$ and covariance matrix $B$.

\medskip

\noindent Our main result is that $\alpha_n$ converges, in total variation a.s., to a r.p.m. $\alpha$. An explicit formula for $\alpha$ is provided and the rate of convergence is shown to be arbitrarily close to $n^{-1/2}$. Moreover, it is still true that $\norm{\alpha_n-\alpha}\overset{a.s.}\longrightarrow 0$ even if $\alpha_n=\mathcal{L}(M_n,Q_n)$ where $\mathcal{L}$ belongs to a class of distributions much larger than $\mathcal{N}$.

\medskip

\noindent In case $S=\mathbb{R}$, a natural competitor to the previous $\alpha_n$ are the copula-based predictive distributions, introduced in \cite{HMW18} and denoted by $\beta_n$ in this paper. Hence, the asymptotic behavior of $\beta_n$ is investigated as well. In Section \ref{x6h8k1q}, it is shown that $\beta_n$ converges weakly a.s. but not necessarily in total variation a.s. Moreover, conditions for $\beta_n$ to converge in total variation a.s. are provided.

\medskip

\noindent Finally, in Section \ref{ane456}, the empirical behavior of the $\alpha_n$ is tested via a numerical experiment. In particular, $\alpha_n$ is compared with $\beta_n$ as regards the speed of convergence. Moreover, both $\alpha_n$ and $\beta_n$ are used to implement the predictive resampling method for parameter estimation.

\medskip

\noindent A last remark is that, to support the $\alpha_n$, we focused mainly on their behaviour in parameter estimation via predictive resampling. However, the $\alpha_n$'s scope is much larger. As an obvious example, in Bayesian predictive inference, the $\alpha_n$ may be used to predict future observations.

\medskip

\section{Results}\label{c5rsd3}
\noindent This section includes our main results. Any consideration regarding their practical application is postponed to Section \ref{ane456}. Similarly, to make the paper more readable, all the proofs are deferred to a final Appendix.

\subsection{Preliminaries and notation}
From now on, $S=\mathbb{R}^p$ where $p\ge 1$. Each point $x\in\mathbb{R}^p$ is regarded as a column vector and $x^{(i)}$ denotes the $i$-th coordinate of $x$. Similarly, if $B$ is any matrix, $B^{(i,j)}$ is the $(i,j)$-th entry of $B$. We denote by $\mathcal{M}$ the collection of symmetric positive definite matrices of order $p\times p$. We write $\mathcal{L}(a,B)$ to indicate the probability distribution of $a+B^{1/2}Z$, where $a\in\mathbb{R}^p$, $B$ is a symmetric non-negative definite matrix of order $p\times p$, and
$Z$ is a $p$-dimensional random vector such that
\begin{gather*}
E(Z)=0,\quad E(ZZ^t)=I,\quad E\bigl\{(Z^tZ)^{u/2}\bigr\}<\infty\quad\text{for some }u>2,\text{ and}
\\Z\,\text{ has a density with respect to Lebesgue measure on }\mathbb{R}^p.
\end{gather*}
In the previous notation, $p$ is understood. We will write $\mathcal{M}_p$ or $\mathcal{L}_p(a,B)$ instead of $\mathcal{M}$ or $\mathcal{L}(a,B)$ every time that a mention of $p$ is necessary.

\medskip

\noindent Finally, given any sequence $(\alpha_n)$ of predictive distributions, convergence is always meant with respect to the probability measure $P$ under which $X$ has predictives $\alpha_n$. As noted in Section \ref{intro}, such a $P$ exists and is unique by the Ionescu-Tulcea theorem.

\medskip

\subsection{Definition and asymptotics of the new predictive distributions}\label{r4es3}
Let
\begin{gather*}
M_n=\frac{1}{n}\,\sum_{i=1}^nX_i,\quad\quad Q_n=\frac{1}{n}\,\sum_{i=1}^n(X_i-M_n)(X_i-M_n)^t,
\end{gather*}
and
\begin{gather}\label{cf332wq1}
\alpha_0=\mathcal{L}(0,I),\quad\alpha_1=\mathcal{L}(X_1,I)\quad\text{and}\quad\alpha_n=\mathcal{L}(M_n,\,Q_n)\quad\text{for }n\ge 2.
\end{gather}

\medskip

\noindent The predictives $\alpha_n$ have a simple form and can be easily handled in real problems. Among other things, they can serve in prediction problems and/or in applying the predictive resampling method; see Section \ref{ane456}. In addition, the $\alpha_n$ have a nice asymptotic behavior. This fact is formalized by the next two results, which are the main contributions of this paper.

\begin{theorem}\label{q2098nfer}
If the $\alpha_n$ are defined by \eqref{cf332wq1}, then $M_n\overset{a.s.}\longrightarrow M$ and $Q_n\overset{a.s.}\longrightarrow Q$ where $M$ is a $p$-dimensional random vector and $Q$ a random matrix of order $p\times p$. In addition, $Q\in\mathcal{M}$ a.s., so that
\begin{gather*}
\norm{\alpha_n-\alpha}\overset{a.s.}\longrightarrow 0\quad\quad\text{where }\alpha=\mathcal{L}(M,Q).
\end{gather*}
\end{theorem}

\medskip

\noindent By Theorem \ref{q2098nfer}, $\alpha_n$ converges in total variation a.s. to $\alpha=\mathcal{L}(M,Q)$, where $M$ and $Q$ are the a.s. limits of the sample means and the sample covariance matrices, respectively. As noted in Section \ref{x4rg7h}, this asymptotic behavior of $\alpha_n$ is useful in various frameworks. To make Theorem \ref{q2098nfer} more effective, however, it is desirable to grasp some information on the convergence rate of $\alpha_n$ to $\alpha$. This is actually possible if $\mathcal{L}(a,B)=\mathcal{N}(a,B)$.

\medskip

\begin{theorem}\label{de4v6}
Let $\alpha_n$, $M$ and $Q$ be as in Theorem \ref{q2098nfer}. Suppose $\mathcal{L}(a,B)=\mathcal{N}(a,B)$, define $\alpha=\mathcal{N}(M,Q)$, and fix any sequence $(d_n)$ of constants. Then,
$$d_n\,\norm{\alpha_n-\alpha}\overset{P}\longrightarrow 0\quad\quad\text{whenever }\,\frac{d_n}{\sqrt{n}}\rightarrow 0.$$
\end{theorem}

\medskip

\noindent \medskip

\noindent To better appreciate Theorem \ref{de4v6}, it is worth noting that $\sqrt{n}\,\norm{\alpha_n-\alpha}$ fails to converge to 0 in probability. We prove this fact for $p=1$. In this case, for each $x\in\mathbb{R}$, one obtains
$$\norm{\alpha_n-\alpha}\ge\Abs{\alpha_n\bigl((-\infty,x]\bigr)-\alpha\bigl((-\infty,x]\bigr)}=\Abs{\Phi\left(\frac{x-M_n}{\sqrt{Q_n}}\right)-\Phi\left(\frac{x-M}{\sqrt{Q}}\right)}$$
where $\Phi$ is the standard normal distribution function. Letting $x=M_n$, it follows that
\begin{gather*}
\sqrt{n}\,\norm{\alpha_n-\alpha}\ge\sqrt{n}\,\,\Abs{\Phi(0)-\Phi\left(\frac{M_n-M}{\sqrt{Q}}\right)}=\frac{\sqrt{n}}{\sqrt{2\pi}}\,\int_0^{\frac{\abs{M_n-M}}{\sqrt{Q}}}e^{-t^2/2}\,dt
\\\ge\frac{1}{\sqrt{2\pi}}\,\,\exp\left(-\frac{(M_n-M)^2}{2Q} \right)\,\frac{\sqrt{n}}{\sqrt{Q}}\,\abs{M_n-M}.
\end{gather*}
However, $\frac{\sqrt{n}}{\sqrt{Q}}\,(M_n-M)$ converges in distribution to the $\mathcal{N}(0,1)$. Hence, the previous inequality implies that $\sqrt{n}\,\norm{\alpha_n-\alpha}$ does not converge to 0 in probability.

\medskip

\subsection{Copula-based predictive distributions}\label{x6h8k1q}

A natural competitor of the predictives $\alpha_n$ of Section \ref{r4es3} are the copula-based predictive distributions of \cite{HMW18}.

\medskip

\noindent A bivariate {\em copula} is a distribution function on $\mathbb{R}^2$ whose marginals are uniform on $(0,1)$. The density of a bivariate copula, provided it exists, is said to be a {\em copula density}. (Henceforth, all densities are meant with respect to Lebesgue measure).

\medskip

\noindent Let $S=\mathbb{R}$. Fix a density $f_0$, a sequence $(c_n)$ of bivariate copula densities, and a sequence $(r_n)$ of constants in $(0,1]$. For the sake of simplicity, assume $f_0>0$ and $c_n>0$ for all $n\ge 0$. Then, a sequence $(\beta_n)$ of predictive distributions may be built as follows. Call $F_0$ the distribution function corresponding to $f_0$ and define
$$f_1(x\mid y)=(1-r_0)\,f_0(x)+r_0\,f_0(x)\,c_0\Bigl(F_0(x),\,F_0(y)\Bigr)\quad\text{for }x,\,y\in\mathbb{R}.$$
Note that $f_1:\mathbb{R}^2\rightarrow\mathbb{R}^+$ is a measurable function such that $f_1(\cdot\mid y)$ is a density for each fixed $y\in\mathbb{R}$. Hence, if $F_1(\cdot\mid y)$ denotes the distribution function corresponding to $f_1(\cdot\mid y)$, one can let
$$f_2(x\mid y,z)=(1-r_1)\,f_1(x\mid y)\,+\,r_1\,f_1(x\mid y)\,c_1\Bigl(F_1(x\mid y),\,F_1(z\mid y)\Bigr)$$
for $x,\,y,\,z\in\mathbb{R}$. In general, given a measurable function $f_n:\mathbb{R}^{n+1}\rightarrow\mathbb{R}^+$ such that $f_n(\cdot\mid y)$ is a density for each fixed $y\in\mathbb{R}^n$, one can define
$$f_{n+1}(x\mid y,z)=(1-r_n)\,f_n(x\mid y)\,+\,r_n\,f_n(x\mid y)\,c_n\Bigl(F_n(x\mid y),\,F_n(z\mid y)\Bigr)$$
for all $x,\,z\in\mathbb{R}$ and $y\in\mathbb{R}^n$. In the above formula, $F_n(\cdot\mid y)$ denotes the distribution function corresponding to $f_n(\cdot\mid y)$. Proceeding in this way, one obtains a collection of densities $\bigl\{f_n(\cdot\mid y):n\ge 1,\,y\in\mathbb{R}^n\bigr\}$ and the $\beta_n$ can be defined as
$$\beta_0(dx)=f_0(x)\,dx\quad\text{and}\quad\beta_n(dx)=f_n\bigl(x\mid X(n)\bigr)\,dx\quad\text{for }n\ge 1.$$

\medskip

\noindent The predictives $\beta_n$ have been introduced in \cite{HMW18} and then used in \cite{FHW23}. The construction given here is from \cite{STS23}.

\medskip

\noindent If the predictive distributions of $X$ are the $\beta_n$, then $X$ is conditionally identically distributed. This is proved in \cite{STS23} when $r_n=1$ but the proof of \cite{STS23} can be easily extended to any choice of $r_n$. Therefore, since $X$ is conditionally identically distributed, there is a r.p.m. $\beta$ such that
$$\beta_n\rightarrow\beta\quad\text{weakly a.s.}$$
A natural question is whether $\beta_n\rightarrow\beta$ in total variation a.s., or equivalently
\begin{gather}\label{m8j9yu7}
\beta(dx)=f(x)\,dx\quad\quad\text{for some random density }f;
\end{gather}
see Example \ref{u8j4z2q1}. Among other things, condition \eqref{m8j9yu7} is tacitly assumed in \cite{FHW23} and \cite{HMW18}. As we now prove, however, it is not necessarily true.

\medskip

\begin{example}\label{bh78u}
Condition \eqref{m8j9yu7} fails whenever
$$\limsup_nr_n>0\quad\text{and}\quad c_n=c\quad\text{for all }n\ge 0$$
where $c$ is any bivariate copula density such that $\int_0^1 \abs{c(u,\,v)-1}\,du>0$ for each $v\in (0,1)$. In fact, suppose the predictives of $X$ are the $\beta_n$ and define
$$D_n=\int_{-\infty}^\infty\,\Abs{f_{n+1}\bigl(x\mid X(n+1)\bigr)-f_n\bigl(x\mid X(n)\bigr)}\,dx.$$
Since $X$ is conditionally identically distributed, condition \eqref{m8j9yu7} amounts to
$$\int_{-\infty}^\infty\,\Abs{f_n\bigl(x\mid X(n)\bigr)-f(x)}\,dx\,\overset{a.s.}\longrightarrow\,0\quad\quad\text{for some random density }f.$$
Hence, under \eqref{m8j9yu7}, one obtains $D_n\overset{a.s.}\longrightarrow 0$. However,
\begin{gather*}
D_n=r_n\,\int_{-\infty}^\infty\,\Abs{c\Bigl(F_n(x\mid X(n)),\,F_n(X_{n+1}\mid X(n))\Bigr)-1}\,\,f_n(x\mid X(n))\,dx
\\=r_n\,\int_0^1 \Abs{c\Bigl(u,\,F_n(X_{n+1}\mid X(n))\Bigr)-1}\,du.
\end{gather*}
Conditionally on $X(n)$, the random variable $F_n(X_{n+1}\mid X(n))$ is uniformly distrubuted on $(0,1)$. Hence, $D_n$ is distributed as
$$r_n\int_0^1 \abs{c(u,\,V)-1}\,du$$
where $V$ denotes any random variable with uniform distribution on $(0,1)$. Since $\int_0^1 \abs{c(u,\,V)-1}\,du>0$ a.s. and $\limsup_nr_n>0$, it follows that $D_n$ does not converge to 0 a.s. Therefore, condition \eqref{m8j9yu7} fails.
\end{example}

\medskip

\noindent Example \ref{bh78u} should be compared with Theorem \ref{q2098nfer}. According to the latter, $\alpha_n\rightarrow\alpha$ in total variation a.s. where the $\alpha_n$ are given by \eqref{cf332wq1} and $\alpha=\mathcal{L}(M,Q)$. Since $\mathcal{L}(0,I)$ admits a density, one obtains $\alpha(dx)=f(x)\,dx$ for some random density $f$.

\medskip

\noindent Even if not generally true, however, condition \eqref{m8j9yu7} holds under some assumptions. We finally prove two results of this type. As probably expected, condition \eqref{m8j9yu7} holds provided $r_n\rightarrow 0$ fastly enough.

\medskip

\begin{theorem}\label{g7j09i}
Condition \eqref{m8j9yu7} holds whenever $\sum_nr_n<\infty$.
\end{theorem}

\medskip

\noindent In the second result, the condition $\sum_nr_n<\infty$ is weakened at the price of a boundedness assumption on the copula densities $c_n$.

\medskip

\begin{theorem}\label{r58uh6}
Condition \eqref{m8j9yu7} holds whenever $\int_{-\infty}^\infty f_0(x)^2dx<\infty$, $\sum_nr_n^2<\infty$, and there is a constant $b$ such that
$$c_n(u,v)\le b\quad\text{for all }n\ge 0\text{ and }u,\,v\in [0,1].$$
\end{theorem}

\medskip

\noindent Theorem \ref{r58uh6} applies to any sequence $c_n$ of bivariate copula densities, but it requires a strong boundedness condition on $c_n$. As expected, the latter condition can be dropped in some special cases. For instance, as proved in \cite[Theorem 5]{FHW23}, it can be dropped if $c_n=c_\rho$ for all $n$ where $c_\rho$ is a Gaussian copula density with correlation $\rho<1/\sqrt{3}$.

\medskip

\section{Numerical illustrations}\label{ane456}

This section reports the results of an empirical comparison between the predictives $\alpha_n$ of Section \ref{r4es3} and the copula-based predictive distributions $\beta_n$. The comparison is based on the speed of convergence and the performance in Bayesian parametric inference when applying the {\em predictive resampling} (PR) method.

\medskip

We let $p = 1$. Hence, $\alpha_0=\mathcal{N}(0,1)$, $\alpha_1=\mathcal{N}(X_1,1)$ and $\alpha_n=\mathcal{N}(M_n,Q_n)$ for $n>1$. In turn, the $\beta_n$ are built as in Section \ref{x6h8k1q} with
$$f_0=\,\text{standard normal density}\quad \text{and} \quad c_n=c_\rho \quad \text{for all} \; n\ge 0$$
where $c_\rho$ is the copula density corresponding to a bivariate normal distribution with mean 0, variance 1 and correlation $\rho\in (0,1)$; see \cite[Example 2]{HMW18}. Note that $\alpha_0=\beta_0=\mathcal{N}(0,1)$. To define the $\beta_n$, we also need a sequence $(r_n)$ of weights. For reasons that are pointed out in Sections \ref{speed_conv} and \ref{par_est}, $(r_n)$ is specified in two different ways.

\medskip

We begin with a brief recap on PR.

\subsection{Predictive resampling}\label{g7uj1}
The PR method, introduced in \cite{FHW23}, allows to make inference on a random parameter and/or to predict future observations. In a Bayesian framework, it is an alternative to the usual likelihood/prior scheme, and it has various similarities with the predictive approach to statistics \cite{STS23}. In particular, the explicit assignment of a prior distribution is not required.

\medskip

\noindent Suppose the object of inference is a function of the sequence $X$, say
$$\theta=f(X)=f(X_1,X_2,\ldots)$$
where $f$ is a known measurable function. As an obvious example, think of $\theta$ as the mean of a population described by $X$, that is
\begin{gather}\label{q3vv7h}
\theta = \lim\limits_{n \to \infty} \dfrac{1}{n} \sum\limits_{i = 1}^n X_i\quad\quad\text{(where the limit is assumed to exist a.s.)}
\end{gather}
More generally, the informal idea is that the uncertainty about $\theta$ would disappear if the whole data sequence $X$ could be observed. The goal is to estimate the distribution of $\theta$ based on the available data.

\medskip

To achieve this target, we first select a sequence $(\alpha_n)$ of predictive distributions. Moreover, we assume that a dataset
$$x=(x_1,\ldots,x_s)$$
has been observed, where $s$ denotes the time when PR is applied. Then, through $(\alpha_n)$ and $x$, we draw $N$ more observations $(X_{s+1},\ldots,X_{s+N})$ and we use the data $(x,X_{s+1},\ldots,X_{s+N})$ to estimate $\theta$. This procedure is repeated a number $B$ of times in order to end up with a sample of estimates of $\theta$. More precisely, let $\hat{\theta} = \hat{\theta}(x,X_{s+1},\ldots,X_{s+N})$ denote the adopted estimate of $\theta$. For instance, if $\theta$ is given by \eqref{q3vv7h}, it is quite natural to let
$$\hat{\theta} =\frac{\sum_{i=1}^sx_i+\sum_{i=1}^NX_{s+i}}{N+s}.$$
Once $\hat{\theta}$ has been chosen, PR can be outlined as follows.

\medskip

\begin{center}
\begin{minipage}{.65\linewidth}
\fbox{\begin{algorithm}[H]
\RestyleAlgo{boxed}
\For{$i = 1,\ldots,B$}{
\For{$j = 1, \ldots , N$}{
Sample $X_{s+j} \sim \alpha_{s+j-1}$ \\
Update: $\{X_{s+j},\alpha_{s+j-1}\} \mapsto \alpha_{s+j}$
}
Evaluate $\theta_i = \hat{\theta}(x,X_{s+1},\ldots,X_{s+N})$
}
Return $\bs{\theta} = (\theta_1, \ldots , \theta_B)$.
\end{algorithm}}
\end{minipage}
\end{center}

\medskip

Note that, since the data $x=(x_1,\ldots,x_s)$ are available, we start sampling from $\alpha_s$ and not from $\alpha_0$. Moreover, each $x_i$ is implicitly assumed to have been drawn from $\alpha_{i-1}$, for $i = 1,\ldots,s$.

\medskip

The final output of PR is a random sample $(\theta_1,\ldots,\theta_B)$ from the probability distribution of $\hat{\theta}(x,X_{s+1},\ldots,X_{s+N})$. In particular, $\theta_1,\ldots,\theta_B$ are i.i.d. and their common distribution $\Pi_N(\cdot\mid x)$ can be written as
$$ \Pi_N(A\mid x) = \int_{\mathbb{R}^N} \mathbbm{1}_A\{\hat{\theta}(x,x_{s+1},\ldots,x_{s+N}) \}\prod\limits_{j=1}^N \alpha_{s+j-1}(x,x_{s+1},\ldots,x_{s+j-1})(dx_{s+j}). $$

\medskip

$\Pi_N(\cdot\mid x)$ can be regarded as the PR analogue of the usual posterior distribution and it is used to approximate the distribution of $\theta$ given $x$, i.e.
$$ \Pi_{\infty}(A\mid x) = P(f \in A\mid x)$$
where $P(\cdot\mid x)$ denotes the probability law induced by the predictives $(\alpha_{n}:n \geq s) $ according to the Ionescu-Tulcea theorem; see Section \ref{intro}.

\medskip

So far, $(\alpha_n)$ is any sequence of predictives. However, for PR to make sense, $\alpha_n$ is required to converge (in some sense) as $n\rightarrow\infty$. Convergence is a sort of consistency condition. Roughly speaking, the population we are trying to reconstruct should mimic the observed sample. Essentially, it should be a version of the sample with similar shape and location but more uncertainty. This informal idea is realized by regarding such a population as distributed according to $\alpha$, where $\alpha$ is the limit of the $\alpha_n$. In addition, it is crucial that $\alpha_s(x)$ fits the observed data $x$ or, at least, shares with $x$ those features that we are estimating. Heuristically, these two requirements (convergence of $\alpha_n$ and good fit to $x$) guarantee that $\Pi_N(\cdot\mid x)$ approximates $\Pi_\infty(\cdot\mid x)$.

\subsection{Speed of convergence}\label{speed_conv}

Let $f_n^\alpha$ and $f_n^\beta$ denote the densities of $\alpha_n$ and $\beta_n$ (with respect to Lebesgue measure). Following \cite{FHW23}, to evaluate the speed of convergence, we computed the $L^1$-distances $\norm{f_0 - f_n^\alpha}_1$ and $\norm{f_0 - f_n^\beta}_1$ where $f_0$ is the standard normal density. Roughly speaking, the informal idea is that, for any sequence $(f_n)$ of random densities, the value of $n$ at which $\norm{f_0 - f_n}_1$ plateaus can be regarded as the index after which $f_n$ is close to its limit (provided it exists).

\medskip

Recall that, thanks to Theorem \ref{q2098nfer}, $\norm{f_n^\alpha-f}_1\overset{a.s.}\longrightarrow 0$ for some random density $f$. As for $\beta_n$, we focus on two different situations. In the first, $\beta_n$ does not necessarily converge in total variation a.s. (at least to our knowledge) while in the second it does, i.e. \begin{itemize}
\item[a)] $\displaystyle{r_n = r^{(a)}_n = \left(2 - \frac{1}{n+1}\right)\frac{1}{n+2}}$ and $\rho = 0.9$.  This choice of $r_n$, inspired by the Dirichlet process mixture model, is proposed in \cite{FHW23}. However, Theorems \ref{g7j09i}-\ref{r58uh6} do not apply and the same happens for Theorem 5 of \cite{FHW23} since $\rho >1/\sqrt{3}$. Hence, whether or not $\beta_n$ converges in total variation a.s. is unknown.
\item[b)] $\displaystyle{r_n = r^{(b)}_n = \left(2 - \dfrac{1}{n+2}\right)\dfrac{1}{(n+3)\,\bigl(\log (n+3)\bigr)^2}}$ and $\rho = 0.9$. In this case, $\sum_n r_n < \infty$ so that $\beta_n$ converges in total variation a.s. by Theorem \ref{g7j09i}.
\end{itemize}

\medskip

\noindent To evaluate $\norm{f_0 - f_n^\alpha}_1$, we ran PR starting from $\alpha_0$ (and not from $\alpha_s(x)$ as in Section \ref{g7uj1}). In other terms, we just sampled $X_1 \sim \alpha_0$, $X_2 \sim \alpha_1(X_1)$, $X_3 \sim \alpha_2(X_1,X_2)$ and so on. Then, at each step of the algorithm, we evaluated $f_n^\alpha$ on a grid of points and computed $\norm{f_0 - f_n^\alpha}_1$ numerically. Exactly the same procedure has been followed for $\norm{f_0 - f_n^\beta}_1$.

\medskip

\noindent Figure \ref{gauss_path} shows the path of $\norm{f_0 - f_n^\alpha}_1$. It takes about 150 iteration of PR for the $L^1$-distance to stabilise.

\vspace{-0.25cm}

\begin{center}
\begin{figure}[H]
\includegraphics[width=8cm, height=7cm]{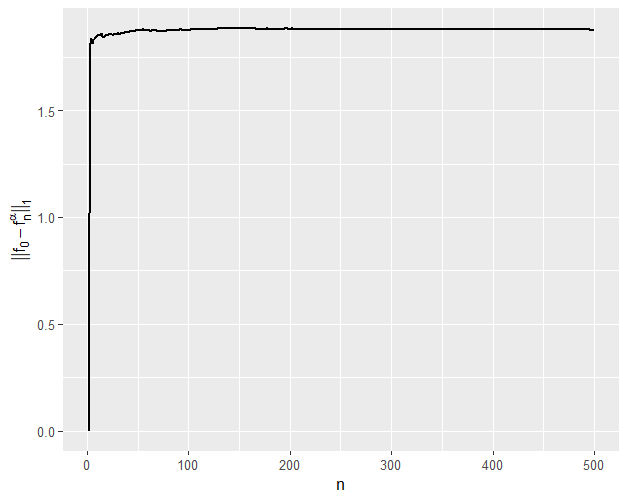}
\caption{$\alpha_n$}
\label{gauss_path}
\end{figure}
\end{center}

\vspace{-1cm}

Figure \ref{fig:cop_paths} shows the path of  $\norm{f_0 - f_n^\beta}_1$ for both $r_n^{(a)}$ and $r_n^{(b)}$ (with $\rho = 0.9$). In the first case, it takes about 4000 iteration of PR for the $L^1$-distance to stabilise, while in the second the behaviour of $\norm{f_0 - f_n^\beta}_1$ is similar to that of $\norm{f_0 - f_n^\alpha}_1$. There might be two reasons for this discrepancy. Either $\beta_n$ fails to converge in total variation a.s. when $r_n = r_n^{(a)}$ and $\rho = 0.9$. Or, $r_n^{(b)}$ approaches 0 so much faster than $r_n^{(a)}$ causing a faster convergence of $||f_0 - f_n^{\beta}||$.

\vspace{-0.25cm}

\begin{figure}[H]
\centering

  \makebox[\textwidth][c]{

  \begin{subfigure}[c]{0.6\textwidth}
    \includegraphics[width=8cm, height=6.5cm]{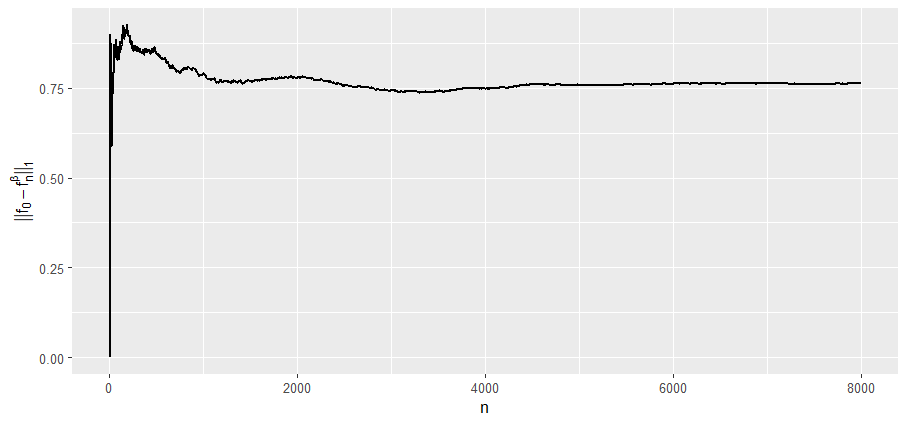}
    \caption{$r_n^{(a)} = \left(2 - \frac{1}{n+1}\right)\frac{1}{n+2}$ and $\rho = 0.9$.}
    \label{fig:cop_path}
  \end{subfigure} \hspace{5mm}
  \begin{subfigure}[c]{0.6\textwidth}
    \includegraphics[width=8cm, height=6.5cm]{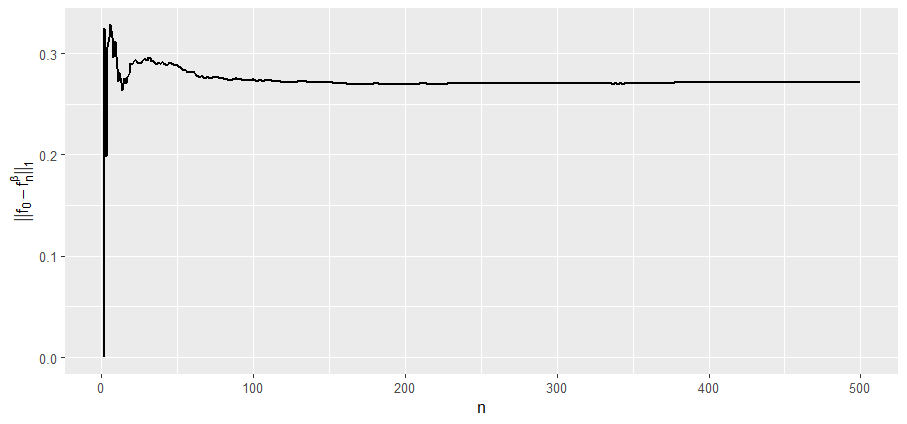}
    \caption{$r_n^{(b)} = \left(2-\frac{1}{n+2}\right)\frac{1}{(n+3)(\log(n+3))^2}$ and $\rho = 0.9$.}
    \label{fig:cop_na_path}
  \end{subfigure}\hspace{3mm}

  }%
  \caption{$\beta_n$}
  \label{fig:cop_paths}
\end{figure}

\subsection{Parameter estimation}\label{par_est}

We now compare $\alpha_n$ and $\beta_n$ when used in PR to estimate the posterior distribution $\Pi_N(\cdot\mid x)$ of a parameter $\theta$; see Section \ref{g7uj1}. Such a $\theta$ is taken to be the mean or the variance of a reference population. In the sequel, $\hat{\theta}$ is short for $\hat{\theta}(x,X_{s+1},\ldots,X_{s+N})$. Obviously, $\hat{\theta}$ is the sample mean of $(x,X_{s+1},\ldots,X_{s+N})$ if $\theta$ is the population mean, while it is the sample variance of $(x,X_{s+1},\ldots,X_{s+N})$ if $\theta$ is the population variance. Moreover, when necessary, we use the notations $\Pi_N^\alpha(\cdot\mid x)$ and $\Pi_N^\beta(\cdot\mid x)$ to stress that $\Pi_N(\cdot\mid x)$ has been obtained using the $\alpha_n$ or the $\beta_n$.

\medskip

The simulation study has been organised as follows. We drew random samples of size $s = 50, 100, 500, 1000, 1500, 2000$ from a mixture of three shifted t-distributions with weights $w_1 = 0.3, w_2 = 0.1 $ and $w_3 = 0.6$, shifting constants $\mu_1 = -5, \mu_2 = 0$ and $\mu_3 = 4$ and degrees of freedom $\nu = 3$. Then, we ran PR with $N = 50, 500, 5000$ and $B = 1000$. Once a sample $\bs{\theta} = (\theta_1,\ldots,\theta_B)$ from $\Pi_N(\cdot\mid x)$ was obtained, we applied kernel density estimation on $\bs{\theta}$ to estimate the posterior density $\pi_N(\cdot\mid x)$ of $\Pi_N(\cdot\mid x)$. We also kept track of the computational times. For the predictives $\beta_n$, we considered both sequences of weights $r_n^{(a)}$ and $r_n^{(b)}$ defined in Section \ref{speed_conv}. $B$ was chosen following \cite{FHW23}. In general, we found $B = 1000$ allows to visualise well the posterior density and does not make PR too slow. As for $s$ and $N$, we decided to consider different values in order to see how $\pi_N(\cdot\mid x)$ varies when more information is available (i.e. $s$ gets larger) and when more uncertainty about the observed sample is reconstructed (i.e. $N$ gets larger). When $\theta$ is the population mean,\linebreak $\pi_N^\beta(\cdot\mid x)$ has been obtained with $r_n = r_n^{(a)}$. Instead, when $\theta$ is the population variance, $\pi_N^\beta(\cdot\mid x)$ has been obtained with $r_n = r_n^{(b)}$. There are no relevant differences between the effect of the two sequences of weights on $\pi_N^\beta(\cdot\mid x)$, but we wanted to provide evidence that both perform well.

\medskip

Figure \ref{fig:mean_posterior} shows the graph of $\pi_N(\cdot\mid x)$ when $\theta$ is the population mean. As a general pattern, note that the posterior variance of $\hat{\theta}$ increases as $N$ gets larger. This might not be clear from the plots because of the different scaling of the axes in the sub-figures corresponding to different values of $N$. Therefore, as an example, in Table \ref{table:var_mean} we report the posterior variance of $\hat{\theta}$ for $s = 2000$. On the other hand, $N$ has no effect on the posterior mean of $\hat{\theta}$. This can already be noted from the plots, but see also Table \ref{table:mean_mean}. Such a behaviour is precisely what we would expect from PR when the predictives are designed as required in Section \ref{g7uj1}, i.e. they converge and fit the observed data $x$. Recall that the goal of PR is to reconstruct uncertainty about the observed sample without changing its shape and location. As shown by Figure \ref{fig:mean_posterior}, this reflects into $\pi_N(\cdot\mid x)$ having larger variance as $N$ increases, but keeping the same mean and number of modes. For $\alpha_n$ this behaviour can be motivated even more explicitly. Indeed, it is quite straightforward to see that
$$ E_N^\alpha(\hat{\theta}\mid x) = \bar{x}_s \quad \text{and} \quad Var_N^\alpha(\hat{\theta}\mid x) = \hat{\sigma}_s^2 \frac{s}{s+1}  \sum_{i = 1}^N \frac{s+i}{(s+i-1)(s+i)^2}, $$ where $\bar{x}_s$ and $\hat{\sigma}_s^2$ are the mean and the variance of the observed dataset $x$. This shows that $\pi_N^\alpha(\cdot\mid x)$ is centred into the sample mean, while its variance increases in $N$.

\medskip

The effect of the sample size $s$ is quite different between $\pi_N^\alpha(\cdot\mid x)$ and $\pi_N^\beta(\cdot\mid x)$. In both cases, the posterior density that best grasps $\theta$, being basically centred on its value, is the one corresponding to $s = 1000$. Nevertheless, for the other values of $s$, the former performs slightly better than the latter. Indeed, $\theta$ always falls in the main body of $\pi_N^\alpha(\cdot\mid x)$, while it is on the tails of $\pi_N^\beta(\cdot\mid x)$ or even quite far from it. Since $E_N^\alpha(\hat{\theta}|x) = \bar{x}_s \overset{a.s.}\longrightarrow \theta$ as $s \to \infty$, using the $\alpha_n$ it is quite natural to expect that estimates are improved by a larger sample size. What is more subtle is the combined effect of $s$ and $N$ on $\pi_N^\alpha(\cdot\mid x)$. Indeed, since a bigger $N$ implies a larger posterior variance, this causes the posterior to catch the true value of $\theta$ even for those values of $s$ for which $\pi_N^\alpha(\cdot\mid x)$ is rather distant from $\theta$ when $N$ is small. As an example, for $s=500$, look at how $\pi_N^\alpha(\cdot\mid x)$ (green curve) changes as $N$ increases.

\medskip

A last remark on Figure \ref{fig:mean_posterior} is in order. As mentioned in Section \ref{g7uj1}, for PR to perform well, the starting predictive distribution (i.e., $\alpha_s(x)$ or $\beta_s(x)$) should fit $x$ well or at least preserve those features that we are interested in estimating. More explicitly, if we want to estimate the posterior distribution of the mean of a population, the predictive mean at time $s$ should be close to $\bar{x}_s$. This is precisely the case of $\alpha_s(x)$. On the contrary, $\beta_s(x)$ is a much more complicated object. In particular, as shown extensively in \cite{FHW23}, $f_s^\beta$ is a very effective density estimator. Hence, Figure \ref{fig:mean_posterior} is an evidence of the fact that to estimate the posterior of a parameter it is enough that the predictive distribution preserves the information about that parameter originating from the observed data. Indeed, both methods provide satisfying estimates (the $\alpha_n$ are even slightly better) and implementing the simpler one is even computationally more efficient. The latter statement will be clarified at the end of this section, but now we can already note that estimating two parameters (i.e. mean and variance) is obviously faster than estimating an entire density.

\medskip

\begin{table}[ht]
    \centering
    \captionsetup{width=0.4\linewidth}
    \begin{minipage}{0.5\linewidth}
        \centering
        \begin{tabular}{|>{\centering\arraybackslash}p{1cm}|>{\centering\arraybackslash}p{1cm}|>{\centering\arraybackslash}p{1cm}|>{\centering\arraybackslash}p{1cm}|}
            \hline
            $N$ & 50 & 500 & 5000 \\ \hline
            $\alpha_n$ & 0.926 & 0.926 & 0.923 \\ \hline
            $\beta_n$ & 0.925 & 0.926 & 0.926 \\ \hline
        \end{tabular}
        \caption{$E_N(\hat{\theta}\mid x)$ for $\theta = \text{population mean}$ and $s = 2000$.}
        \label{table:mean_mean}
    \end{minipage}%
    \hfill
    \begin{minipage}{.5\linewidth}
        \centering
        \begin{tabular}{|>{\centering\arraybackslash}p{1cm}|>{\centering\arraybackslash}p{1cm}|>{\centering\arraybackslash}p{1cm}|>{\centering\arraybackslash}p{1cm}|}
            \hline
            $N$ & 50 & 500 & 5000 \\ \hline
            $\alpha_n$ & 0.0002 & 0.0019 & 0.0068 \\ \hline
            $\beta_n$ & 0.0002 & 0.0015 & 0.0019 \\ \hline
        \end{tabular}
        \caption{$Var_N(\hat{\theta}\mid x)$ for $\theta = \text{population mean}$ and $s = 2000$.}
        \label{table:var_mean}
    \end{minipage}
\end{table}
\vspace{-1mm}
Figure \ref{fig:var_posterior} shows the posterior distribution when $\theta$ is the population variance. The effect of $N$ and $s$ on $\pi_N(\cdot\mid x)$ is basically the same as in Figure \ref{fig:mean_posterior}, as much as the comparison between $\pi_N^\alpha(\cdot\mid x)$ and $\pi_N^\beta(\cdot\mid x)$. Even in this case, the former is more precise in catching the value of $\theta$. We just mention an additional detail about variance estimation when using the $\alpha_n$. In this case, one obtains
$$ E_N^\alpha(\hat{\theta}\mid x) = \hat{\sigma}_s^2 \frac{s}{s+1} \frac{s+N+1}{s+N} $$
and
$$ Var_N^\alpha(\hat{\theta}\mid x) = \hat{\sigma}_s^4 \left\{\prod_{i=1}^N \left[1 + \frac{3}{(s+i)^4} - \frac{4}{(s+i)^3} \right] - \left[ \frac{s}{(s+1)} \frac{s+N+1}{s+N} \right]^2 \right\}. $$
Therefore, $\pi_N^\alpha(\cdot\mid x)$ is still centered in the sample statistic of the parameter of interest (with some bias that vanishes as $s \to \infty$) and its dispersion increases as $N \to \infty$ (at least numerically). 

\begin{figure}[H]
\centering

  \makebox[\textwidth][c]{

  \begin{subfigure}[c]{0.6\textwidth}
    \includegraphics[width=8cm, height=6.5cm]{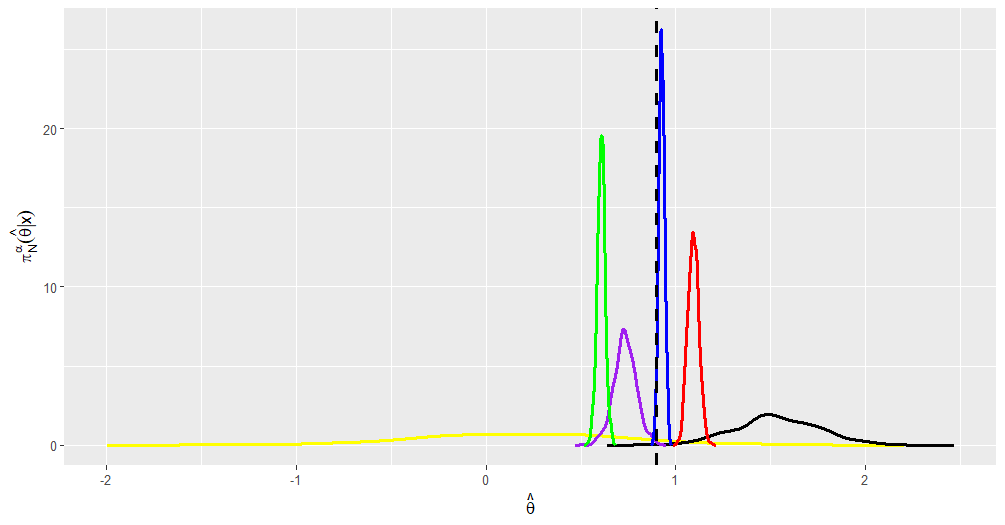}
    \caption{$\pi_N^\alpha(\cdot\mid x), N = 50$.}
    \label{fig:alpha_mean_50}
    \includegraphics[width=8cm, height=6.5cm]{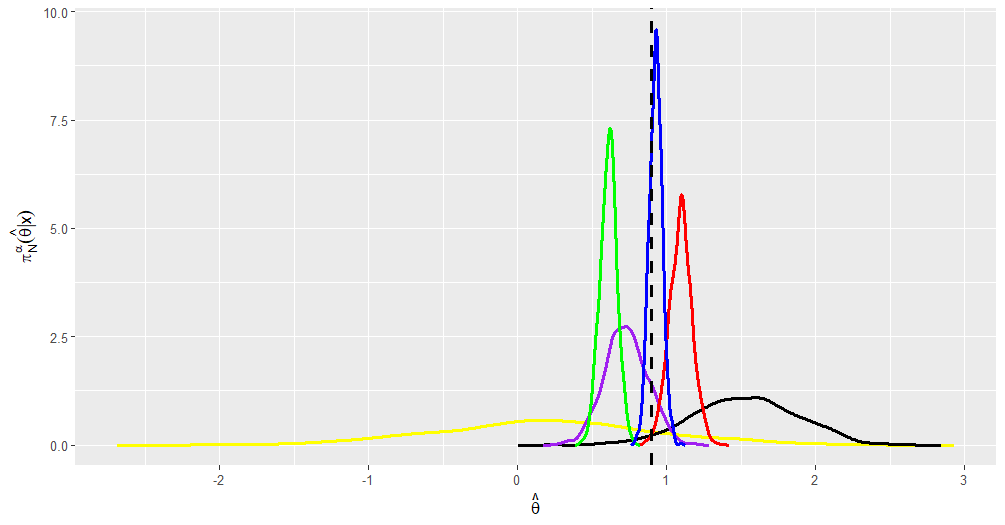}
    \caption{$\pi_N^\alpha(\cdot\mid x), N = 500$.}
    \label{fig:alpha_mean_500}
    \includegraphics[width=8cm, height=6.5cm]{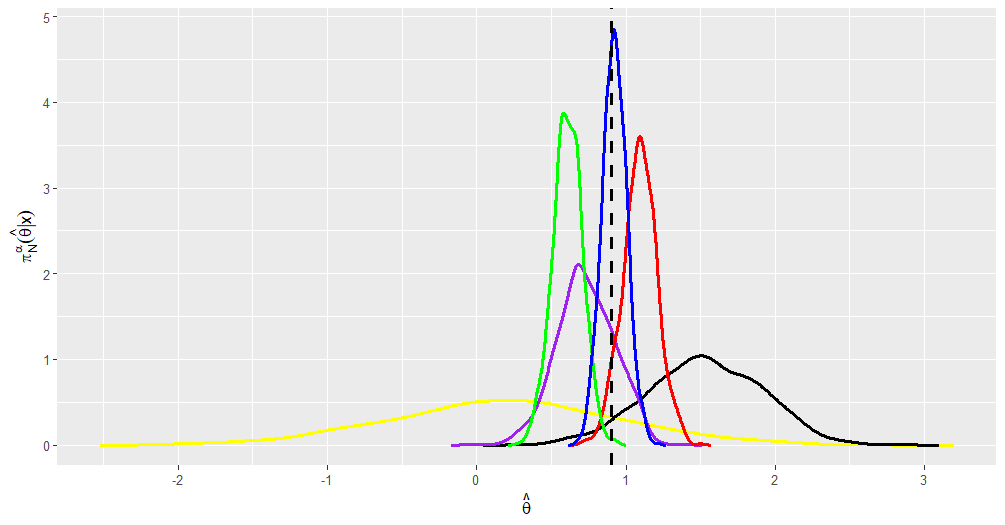}
    \caption{$\pi_N^\alpha(\cdot\mid x), N = 5000$.}
    \label{fig:alpha_mean_5000}
  \end{subfigure} \hspace{5mm}
  \begin{subfigure}[c]{0.6\textwidth}
    \includegraphics[width=8cm, height=6.5cm]{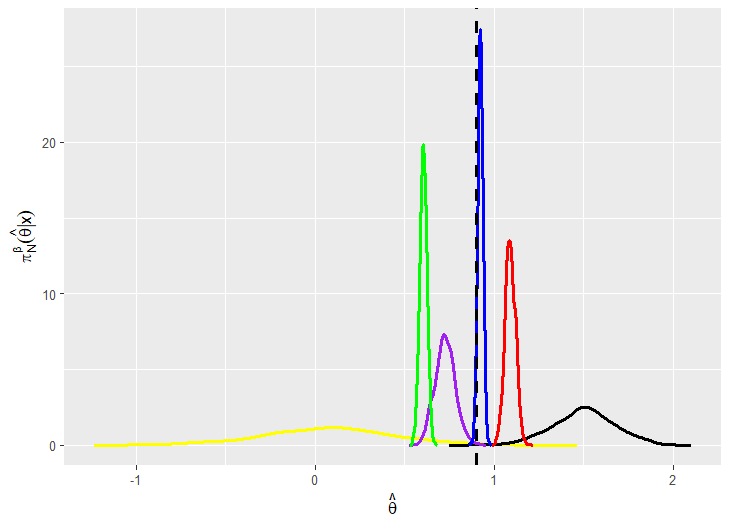}
    \caption{$\pi_N^\beta(\cdot\mid x), r_n^{(a)}, N = 50$.}
    \label{fig:beta_mean_50}
    \includegraphics[width=8cm, height=6.5cm]{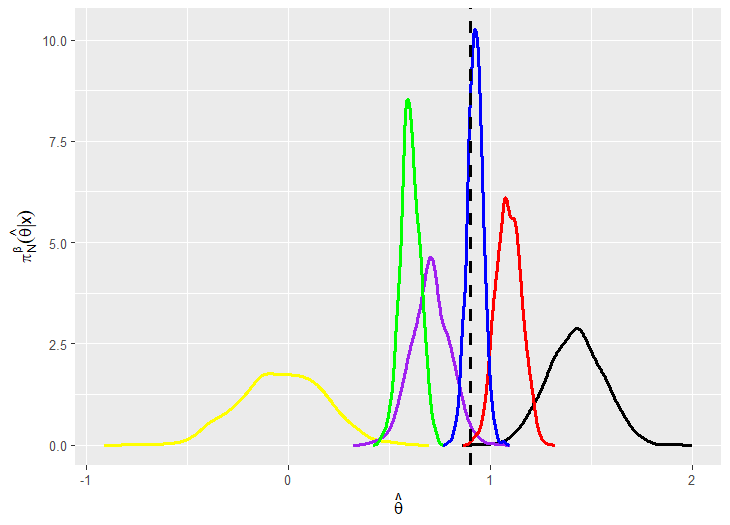}
    \caption{$\pi_N^\beta(\cdot\mid x),r_n^{(a)}, N = 500$.}
    \label{fig:beta_mean_500}
    \includegraphics[width=8cm, height=6.5cm]{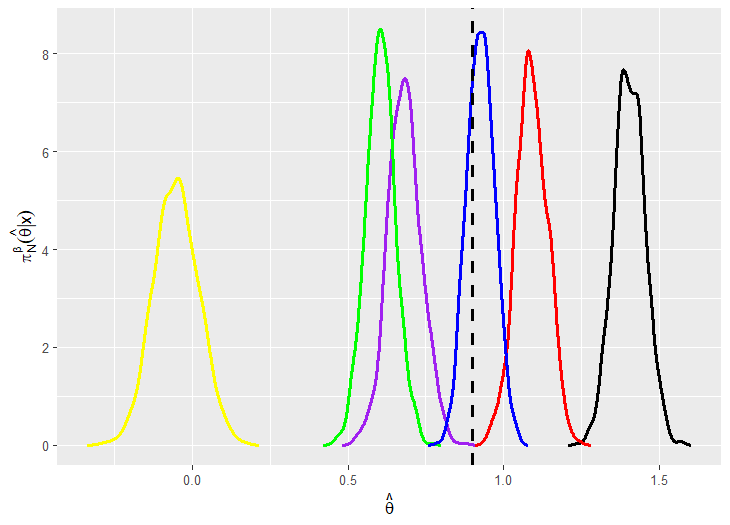}
    \caption{$\pi_N^\beta(\cdot\mid x), r_n^{(a)}, N = 5000$.}
    \label{fig:beta_mean_5000}
  \end{subfigure}\hspace{3mm}

 }%
  \caption{$\pi_N(\cdot\mid x)$ for $\theta = \text{population mean} = 0.9 $ and $s = 50 \ (\textcolor{black}{\rule[0.5ex]{1em}{1pt}}), 100 \ (\textcolor{red}{\rule[0.5ex]{1em}{1pt}}), 500 \ (\textcolor{green}{\rule[0.5ex]{1em}{1pt}}), 1000 \ (\textcolor{blue}{\rule[0.5ex]{1em}{1pt}}), 1500 \ (\textcolor{yellow}{\rule[0.5ex]{1em}{1pt}}), 2000 \ (\textcolor{violet}{\rule[0.5ex]{1em}{1pt}})$.}
  \label{fig:mean_posterior}
\end{figure}

\begin{figure}[H]
\centering

  \makebox[\textwidth][c]{

  \begin{subfigure}[c]{0.6\textwidth}
    \includegraphics[width=8cm, height=6.5cm]{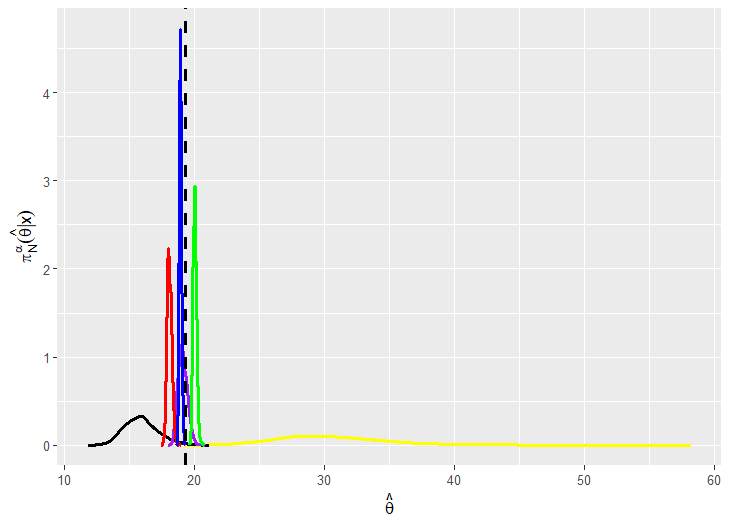}
    \caption{$\pi_N^\alpha(\cdot\mid x), N = 50$.}
    \label{fig:alpha_var_50}
    \includegraphics[width=8cm, height=6.5cm]{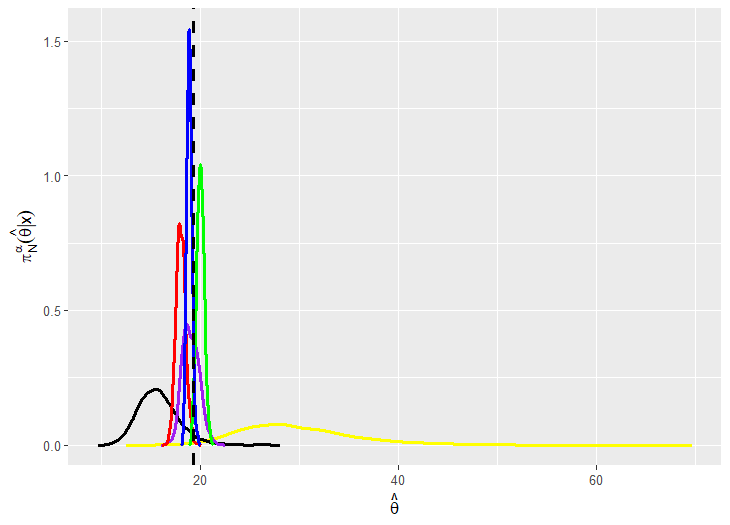}
    \caption{$\pi_N^\alpha(\cdot\mid x), N = 500$.}
    \label{fig:alpha_var_500}
    \includegraphics[width=8cm, height=6.5cm]{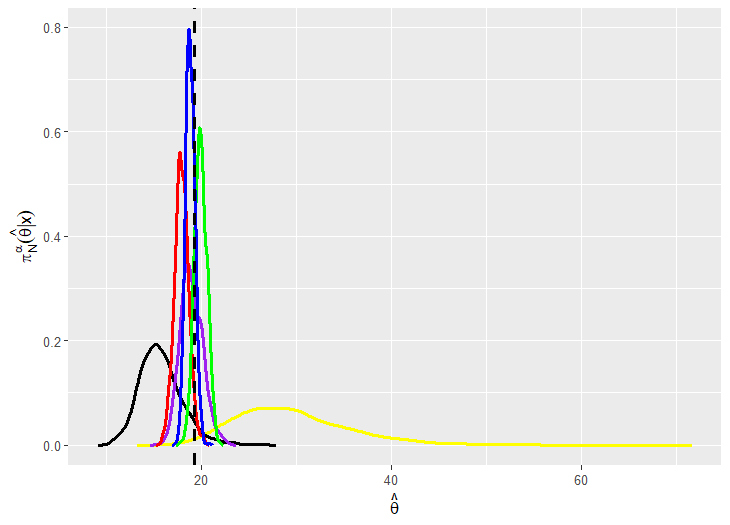}
    \caption{$\pi_N^\alpha(\cdot\mid x), N = 5000$.}
    \label{fig:alpha_var_5000}
  \end{subfigure} \hspace{5mm}
  \begin{subfigure}[c]{0.6\textwidth}
    \includegraphics[width=8cm, height=6.5cm]{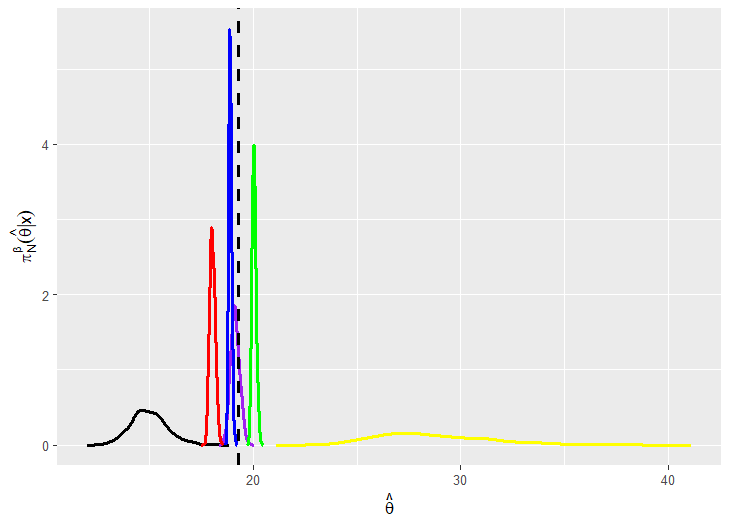}
    \caption{$\pi_N^\beta(\cdot\mid x), r_n^{(b)}, N = 50$.}
    \label{fig:beta_var_50}
    \includegraphics[width=8cm, height=6.5cm]{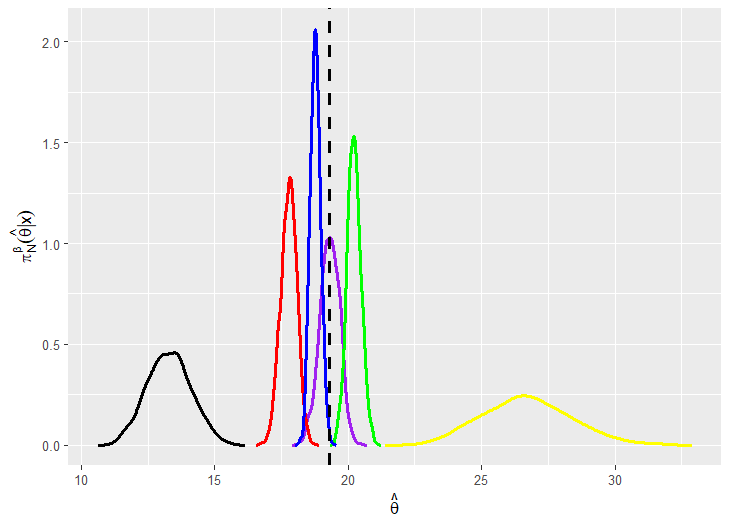}
    \caption{$\pi_N^\beta(\cdot\mid x),r_n^{(b)}, N = 500$.}
    \label{fig:beta_var_500}
    \includegraphics[width=8cm, height=6.5cm]{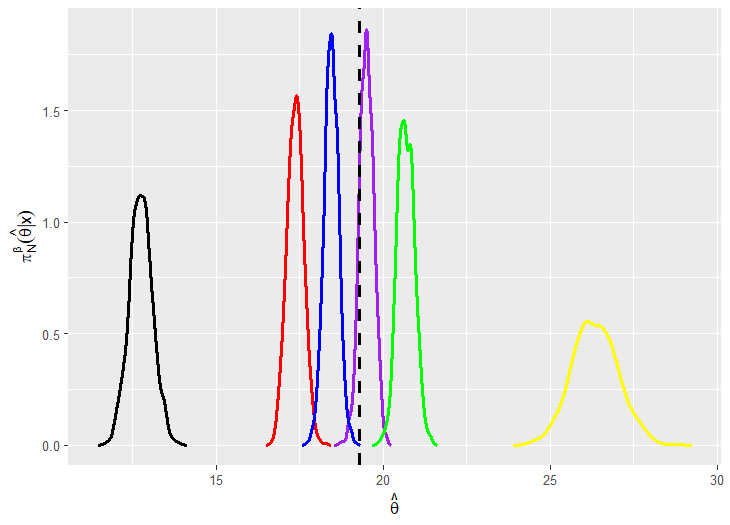}
    \caption{$\pi_N^\beta(\cdot\mid x), r_n^{(b)}, N = 5000$.}
    \label{fig:beta_var_5000}
  \end{subfigure}\hspace{3mm}

 }%
  \caption{$\pi_N(\cdot\mid x)$ for $\theta = \text{population variance} = 19.2 $ and $s = 50 \ (\textcolor{black}{\rule[0.5ex]{1em}{1pt}}), 100 \ (\textcolor{red}{\rule[0.5ex]{1em}{1pt}}), 500 \ (\textcolor{green}{\rule[0.5ex]{1em}{1pt}}), 1000 \ (\textcolor{blue}{\rule[0.5ex]{1em}{1pt}}), 1500 \ (\textcolor{yellow}{\rule[0.5ex]{1em}{1pt}}), 2000 \ (\textcolor{violet}{\rule[0.5ex]{1em}{1pt}})$.}
  \label{fig:var_posterior}
\end{figure}

We next make some general comments on the posterior estimates obtained with $\alpha_n$ and $\beta_n$. As highlighted by Figures \ref{fig:mean_posterior} and \ref{fig:var_posterior}, $\pi_N^\alpha(\cdot\mid x)$ is usually more precise than $\pi_N^\beta(\cdot\mid x)$. The reason may be that in the first case we can easily control the centre of the posterior distribution. Nevertheless, for parameters different from mean and variance (e.g. skewness, kurtosis, quantiles, etc.) $\beta_n$ is expected to provide better estimates. Indeed, $\alpha_n$ does not preserve information from the observed sample about any parameter other than mean and variance. On the contrary, since $\beta_n$ is essentially a density estimator, it grasps the shape of the whole distribution. Furthermore, according to us, the performance of $\beta_n$ could be enhanced in two (obvious) ways. Firstly, the correlation coefficient $\rho \in (0,1)$, involved in the Gaussian copula density $c_\rho$, has to be estimated. To do so, following \cite{FHW23}, we optimised the prequential log-likelihood for $\rho \in \{0.05, 0.1, \ldots, 0.95\}$. By taking a finer grid, $\beta_n$ might have a better fit on the observed data and this might make estimates more precise. The second possible improvement is related to the mechanism adopted to sample from $\beta_n$. Indeed, $X_{s+1},\ldots,X_{s+N}$ are generated via inverse sampling from a grid $\{y_1,\ldots,y_G\}$. Thus, once again, more accurate posterior estimates could be obtained by refining such grid. It is worth stressing that, if one wants only to update $\beta_n$ into $\beta_{n+1}$, it is not necessary to sample $X_{n+1} \sim \beta_n$ and plug it in $\beta_{n+1}$. As explained in \cite[Sec. 4.2]{FHW23}, it is enough to draw an observation from a uniform distribution on $(0,1)$. But clearly, we cannot use this shortcut for we need to know the actual value of the sampled variables. Finally, we note that these improvements come at the cost of longer computational times. We explore this further in the next paragraph.

\medskip

Tables \ref{table:time_alpha} and \ref{table:time_beta} show the time needed to perform PR for each value of $s$ and $N$. Algorithms for the implementation of PR have been coded in Rcpp. Computations have been carried out on a single-core version of R installed on a machine with an AMD Epyc 7763 2.45Ghz processor having 2TB RAM. As expected, in both cases PR is slower as $N$ and $s$ increase. The difference between the computational times of $\alpha_n$ and $\beta_n$, with the former being much faster than the latter, is quite evident. Reasons for this discrepancy have partially been pointed out already, so we summarise them and provide some more details. Sampling from $\alpha_n$ requires just to compute a mean and a variance and then draw an observation from a Gaussian distribution, for which efficient algorithms are available in any statistical software. On the other hand, sampling from $\beta_n$ entails estimating the hyperparameter $\rho$ and doing inverse sampling. The time needed for these tasks depends on the size of the grid of values taken to perform them. Such a grid cannot be too coarse, otherwise the accuracy of the estimates is worsened. Therefore, these two steps are likely to be the main reason why $\alpha_n$ is faster than $\beta_n$. Furthermore, the estimation of $\rho$ is affected by the sample size $s$. Indeed, as explained in \cite{FHW23}, optimizing the prequential log-likelihood requires to evaluate recursively $f_{i-1}^\beta(x_i\mid x_1,\ldots,x_{i-1})$, for each $i = 1,\ldots,s$ and each candidate value of $\rho$, and then pick the one that maximises $\sum_{i = 1}^s\log f_{i-1}^\beta(x_i\mid x_1,\ldots,x_{i-1})$. This procedure is slower the larger the value of $s$ is and has to be carried out before starting PR, because we cannot sample from $\beta_n$ without specifying $\rho$. On the other hand, sampling from $\alpha_n$ does not require any such complicated initialisation step. Therefore, $s$ has a much stronger effect on $\beta_n$ than it has on $\alpha_n$. To see this, look for instance at the increase in computational time from $s = 50$ to $s = 2000$, when $N = 5000$. In this case, PR with $\alpha_n$ is about 14 seconds slower while with $\beta_n$ it is more than 2 minutes slower.

\begin{table}[H]
\centering
\begin{tabular}{|c|c|c|c|c|c|c| }
 \cline{2-7}
 \multicolumn{1}{c|}{} & $s = 50$ & $s = 100$ & $s = 500$ & $s = 1000$ & $s = 1500$ & $s = 2000$ \\
  \hline
 $N = 50$ & 0.0004 & 0.0005 & 0.001 & 0.0016 & 0.0032 & 0.0029 \\
 \hline
 $N = 500$ & 0.0074 & 0.0081 & 0.0133 & 0.0179 & 0.0241 & 0.0299 \\
 \hline
$N = 5000$ & 0.3431 & 0.3498 & 0.4011 & 0.4591 & 0.521 & 0.5818 \\
 \hline
\end{tabular}
\caption{Computational times for $\alpha_n$ (in minutes).}
\label{table:time_alpha}
\end{table}

\begin{table}[H]
\centering
\begin{tabular}{|c|c|c|c|c|c|c| }
 \cline{2-7}
 \multicolumn{1}{c|}{} & $s = 50$ & $s = 100$ & $s = 500$ & $s = 1000$ & $s = 1500$ & $s = 2000$ \\
  \hline
 $N = 50$ & 0.0852 & 0.0988 & 0.3316 & 0.8326 & 1.5237 & 2.3548 \\
 \hline
 $N = 500$ & 0.6719 & 0.6313 & 0.8787 & 1.4696 & 2.1769 & 3.0014 \\
 \hline
$N = 5000$ & 6.547 & 5.9677 & 6.0797 & 7.0695 & 7.9625 & 8.6769 \\
 \hline
\end{tabular}
\caption{Computational times for $\beta_n$ (in minutes).}
\label{table:time_beta}
\end{table}

\section{Discussion}

Let $(X_n)$ be a sequence of $p$-dimensional random vectors and
$$\alpha_n(\cdot)=P\bigl(X_{n+1}\in\cdot\mid X_1,\ldots,X_n\bigr)$$
the corresponding predictive distributions. In this paper, a new sequence $(\alpha_n)$ of predictives is introduced. The main feature of $\alpha_n$ is to depend on $(X_1,\ldots,X_n)$ only through their mean and covariance matrix. Convergence in total variation of $\alpha_n$ is proved, along with an explicit formula for the limit r.p.m. In the (meaningful) special case where the $\alpha_n$ are Gaussian, the convergence rate is shown to be $n^{-1/2}$. A natural competitor of $\alpha_n$ is the copula-based predictive distribution $\beta_n$ of \cite{HMW18}. Thus, $\beta_n$ is investigated as well. It turns out that $\beta_n$ converges weakly but not necessarily in total variation. Hence, sufficient conditions for $\beta_n$ to converge in total variation are provided. Eventually, an empirical comparison between $\alpha_n$ and $\beta_n$ is performed, assessing their speed of convergence and their ability to estimate, via PR, the posterior distribution of the mean and the variance of a reference population.

\medskip

In all the numerical experiments, the kernel of $\alpha_n$ (denoted $\mathcal{L}$) is Gaussian. One reason is that, to estimate the posterior of a parameter $\theta$ via PR, it is enough for the predictive distribution to preserve the information about $\theta$ contained in the observed data. Hence, we chose $\mathcal{L} = \mathcal{N}$. In fact, when $\theta$ is the mean or the variance, it is straightforward to incorporate such information into a Gaussian distribution. Most probably, for estimating other parameters (e.g. skewness, kurtosis, quantiles, etc.) this simple model is no longer appropriate. Perhaps, the $\alpha_n$ should be enabled to capture the whole distribution of the observed data. Roughly speaking, a version of $\alpha_n$ more similar to $\beta_n$ is likely to work better. A natural example could be to set $\mathcal{L}$ as a mixture of Gaussian distributions, each of them centred in the main clusters of the observed data.

\medskip

Moving forward from the idea of taking $\mathcal{L}$ as a mixture of Gaussians, it would be reasonable to set $\alpha_0$ as a mixture distribution (possibly data-dependent) but then update only the mean and the variance of the components from which the observations are drawn (along with the corresponding weights). It would be interesting to study the asymptotics of the resulting sequence of predictives. In addition, these predictives might result into more precise estimates, since the starting distribution and the updating mechanism preserve more of the information contained in the observed data than just their mean and variance. Hence, they would be relevant even from a statistical point of view.

\medskip

We finally mention two further issues. The first is to give conditions on $\mathcal{L}$ under which the convergence rate of $\alpha_n$ is still $n^{-1/2}$ (as it happens when $\mathcal{L}=\mathcal{N}$). The second is using the $\alpha_n$ in regression problems (as done in \cite{FHW23} with the $\beta_n$). In this case, the data are of the type $X_n=(Y_n,Z_n)$, where $Y_n$ is a response variable and $Z_n$ a vector of covariates. Then, the predictives $\alpha_n$ could be applied to estimate, via PR, the conditional distribution of the response given the observed values of the covariates.

\medskip

\begin{center}\textbf{APPENDIX}\end{center}

\bigskip

\noindent We close the paper by proving the results of Section \ref{c5rsd3}. To this end, we first recall two known facts and we make a preliminary remark.

\begin{lemma}\label{ji87c4}
If $x,\,y\in\mathbb{R}^p$ and $B$ is a real non-singular matrix of order $p\times p$, then
$$\text{det}\bigl(B+xy^t)=\text{det}(B)\bigl(1+x^tB^{-1}y\bigr).$$
\end{lemma}

\medskip

\begin{lemma}\label{7yu8}
If $(V_n:n\ge 1)$ is an i.i.d. sequence of real random variables such that $E(V_1)=0$ and $E\bigl(\abs{V_1}^{1+\epsilon}\bigr)<\infty$, for some $\epsilon>0$, then
$$\text{the series }\,\sum_{n=1}^\infty\frac{V_n}{n}\,\text{ converges a.s.}$$
\end{lemma}

\begin{proof}
Since
$$\sum_nP(\abs{V_n}>n)=\sum_nP(\abs{V_1}>n)\le E\bigl(\abs{V_1}\bigr)<\infty,$$
the Borel-Cantelli lemma yields
$$P(\abs{V_n}>n\text{ for infinitely many }n)=0.$$
Hence, it suffices to show that $\sum_n\frac{W_n}{n}$ converges a.s., where
$$W_n=V_n\,1_{\{\abs{V_n}\le n\}}.$$
To this end, first note that $\frac{W_n-E(W_n)}{n}$ are independent, centered, and
$$\sum_nE\left\{\left(\frac{W_n-E(W_n)}{n}\right)^2\right\}\le\sum_nE\left\{\frac{V_1^2\,1_{\{\abs{V_1}\le n\}}}{n^2}\right\}\le\sum_n\frac{E\bigl(\abs{V_1}^{1+\epsilon}\bigr)}{n^{1+\epsilon}}<\infty.$$
Hence, the series $\sum_n\frac{W_n-E(W_n)}{n}$ converges a.s. Finally, $E(V_1)=0$ implies
$$E(W_n)=E\bigl\{V_1\,1_{\{\abs{V_1}\le n\}}\bigr\}=-E\bigl\{V_1\,1_{\{\abs{V_1}> n\}}\bigr\}.$$
Therefore,
$$\sum_n\frac{\abs{E(W_n)}}{n}\le\sum_n\frac{E\bigl\{\abs{V_1}\,1_{\{\abs{V_1}> n\}}\bigr\}}{n}\le\sum_n\frac{E\bigl(\abs{V_1}^{1+\epsilon}\bigr)}{n^{1+\epsilon}}<\infty.$$
This concludes the proof.
\end{proof}

\medskip

\noindent Both the previous lemmas are well known. A possible reference for Lemma \ref{ji87c4} is \cite{DZ07}, while Lemma \ref{7yu8} has been proved for we do not know of any explicit reference. Incidentally, in Lemma \ref{7yu8}, the condition $E\bigl(\abs{V_1}^{1+\epsilon}\bigr)<\infty$ can be possibly weakened but not completely removed. In fact, as shown in the next example, the series $\sum_{n=1}^\infty\frac{V_n}{n}$ may fail to converge a.s. even if $(V_n)$ is i.i.d. and $E(V_1)=0$.

\medskip

\begin{example}
Let $(V_n)$ be i.i.d. with
$$P\bigl(V_1=-c/\log 2\bigr)=1/2\quad\text{and}\quad P(2<V_1\le x)=\frac{c}{2}\,\int_2^x\frac{dt}{(t\,\log t)^2}\quad\text{for }x>2,$$
where the constant $c$ is given by $c=\Bigl(\int_2^\infty\frac{dt}{(t\,\log t)^2}\Bigr)^{-1}$. Then,
$$E(V_1)=-\left\{\frac{c}{2\,\log 2}+\frac{c}{2}\right\},\int_2^\infty\frac{t}{(t\,\log t)^2}\,dt=0.$$
Letting
$$W_n=\frac{V_n}{n}\,1_{\{\abs{V_n}\le n\}},$$
a direct calculation shows that $\sum_nE(W_n)=-\infty$ and $\sum_n\text{Var}(W_n)<\infty$. Hence,
$$\sum_{n=1}^\infty W_n\overset{a.s.}=-\infty.$$
In turn, by the Borel-Cantelli Lemma, this implies $\sum_{n=1}^\infty\frac{V_n}{n}\overset{a.s.}=-\infty.$
\end{example}

\noindent Finally, a remark is in order.

\medskip

\begin{remark}\label{vf67mju8}
Fix an i.i.d. sequence $Z_1,Z_2,\ldots$ of $p$-dimensional random vectors such that $Z_1\sim\mathcal{L}(0,I)$. Define a sequence $Y=(Y_1,Y_2,\ldots)$ as
\begin{gather*}
Y_1=Z_1,\quad Y_2=Y_1+Z_2\quad\text{and}\quad Y_{n+1}=N_n+R_n^{1/2}\,\,Z_{n+1}\quad\text{for }n\ge 2
\\\text{where}\quad N_n=\frac{1}{n}\,\sum_{i=1}^nY_i\quad\text{and}\quad R_n=\frac{1}{n}\,\sum_{i=1}^n(Y_i-N_n)(Y_i-N_n)^t.
\end{gather*}
It is straightforward to see that
\begin{gather*}
P(Y_1\in\cdot)=\mathcal{L}(0,I),\quad P(Y_2\in\cdot\mid Y_1)=\mathcal{L}(Y_1,I),\quad\text{and}
\\P(Y_{n+1}\in\cdot\mid Y_1,\ldots,Y_n)=\mathcal{L}(N_n,\,R_n)\quad\quad\text{for }n\ge 2.
\end{gather*}
Hence, $Y\sim X$ whenever the predictive distributions of $X$ are the $\alpha_n$ of Theorems \ref{q2098nfer} and \ref{de4v6}. In particular, since $Y\sim X$, we can use $Y$ instead of $X$ when proving Theorems \ref{q2098nfer} and \ref{de4v6}. Equivalently, in the proofs of Theorems \ref{q2098nfer} and \ref{de4v6}, we can suppose
\begin{gather*}
X_1=Z_1,\quad X_2=X_1+Z_2\quad\text{and}\quad X_{n+1}=M_n+Q_n^{1/2}\,\,Z_{n+1}\quad\text{for }n\ge 2
\end{gather*}
where  $Z_1,Z_2,\ldots$ are i.i.d. and $Z_1\sim\mathcal{L}(0,I)$.
\end{remark}

\medskip

\noindent We are now able to attack the results of Section \ref{r4es3}. To this end, we let
$$\mathcal{F}_n=\sigma(Z_1,\ldots,Z_n)\quad\text{with}\quad\mathcal{F}_0=\,\text{the trivial $\sigma$-field.}$$

\medskip

\begin{proof}[\textbf{Proof of Theorem \ref{q2098nfer}}]
Suppose $M_n\overset{a.s.}\longrightarrow M$, $Q_n\overset{a.s.}\longrightarrow Q$ and $Q\in\mathcal{M}$ a.s. Let $\phi$ be a density of $\mathcal{L}(0,I)$ with respect to $p$-dimensional Lebesgue measure. Since $Q_n\in\mathcal{M}$ a.s. (as it is easily proved), one obtains
\begin{gather*}
\norm{\alpha_n-\alpha}=\norm{\mathcal{L}(M_n,Q_n)-\mathcal{L}(M,Q)}
\\\overset{a.s.}=\frac{1}{2}\,\int_{\mathbb{R}^p}\Abs{\frac{\phi\bigl(Q_n^{-1/2}(x-M_n)\bigr)}{\text{det}(Q_n)^{1/2}}-\frac{\phi\bigl(Q^{-1/2}(x-M)\bigr)}{\text{det}(Q)^{1/2}}}\,dx.
\end{gather*}
Given $\epsilon>0$, take an integrable continuous function $f$ on $\mathbb{R}^p$ such that
$$\int_{\mathbb{R}^p}\abs{\phi(x)-f(x)}\,dx<\epsilon.$$
Define
\begin{gather*}
I_n=\int_{\mathbb{R}^p}\Abs{\frac{f\bigl(Q_n^{-1/2}(x-M_n)\bigr)}{\text{det}(Q_n)^{1/2}}-\frac{f\bigl(Q^{-1/2}(x-M)\bigr)}{\text{det}(Q)^{1/2}}}\,dx.
\end{gather*}
Since $f$ is continuous, $I_n\overset{a.s.}\longrightarrow 0$. Therefore,
\begin{gather*}
\limsup_n\,\norm{\alpha_n-\alpha}\le\epsilon+\frac{1}{2}\,\limsup_nI_n\overset{a.s.}=\epsilon.
\end{gather*}
This proves that $\norm{\alpha_n-\alpha}\overset{a.s.}\longrightarrow 0$ provided $M_n\overset{a.s.}\longrightarrow M$, $Q_n\overset{a.s.}\longrightarrow Q$ and $Q\in\mathcal{M}$ a.s.

\medskip

\noindent After noting this fact, the rest of the proof is split into three steps. Recall that, by Remark \ref{vf67mju8}, we can assume
$$X_1=Z_1,\quad X_2=X_1+Z_2\quad\text{and}\quad X_{n+1}=M_n+Q_n^{1/2}\,\,Z_{n+1}\quad\text{for }n\ge 2$$
where  $Z_1,Z_2,\ldots$ are i.i.d. and $Z_1\sim\mathcal{L}(0,I).$

\medskip

\noindent\textbf{Step 1: $Q_n$ converges a.s.}

\medskip

\noindent Define
$$M_0=0,\quad Q_0=Q_1=I\quad\text{and}\quad L_n=Q_n^{1/2}Z_{n+1}.$$
Then,
$$M_{n+1}-M_n=\frac{Q_n^{1/2}\,Z_{n+1}}{n+1}=\frac{L_n}{n+1}.$$
After some algebra, this implies
\begin{gather}\label{bh87f4r}
Q_{n+1}=\frac{n}{n+1}\,Q_n+\frac{n}{(n+1)^2}\,L_nL_n^t
\end{gather}
or equivalently
$$Q_{n+1}^{(i,j)}=\frac{n}{n+1}\,Q_n^{(i,j)}+\frac{n}{(n+1)^2}\,L_n^{(i)}\,L_n^{(j)}\quad\quad\text{for all }i,\,j=1,\ldots,p.$$
The conditional distribution of $L_n$ given $\mathcal{F}_n$ is $\mathcal{L}(0,Q_n)$. Therefore,
\begin{gather*}
E\bigl\{L_n^{(i)}\,L_n^{(j)}\mid\mathcal{F}_n\bigr\}=Q_n^{(i,j)}
\end{gather*}
and
\begin{gather*}
E\bigl\{Q_{n+1}^{(i,j)}\mid\mathcal{F}_n\bigr\}=\frac{n}{n+1}\,Q_n^{(i,j)}+\frac{n}{(n+1)^2}\,Q_n^{(i,j)}=Q_n^{(i,j)}\,\left(1-\frac{1}{(n+1)^2}\right).
\end{gather*}
For $i=j$, it follows that $(Q_n^{(i,i)}:n\ge 2)$ is a non-negative supermartingale. Hence,
$$E\bigl\{\abs{Q_n^{(i,j)}}\bigr\}\le\frac{E\bigl\{Q_n^{(i,i)}\bigr\}+E\bigl\{Q_n^{(j,j)}\bigr\}}{2}\le\frac{E\bigl\{Q_2^{(i,i)}\bigr\}+E\bigl\{Q_2^{(j,j)}\bigr\}}{2}$$
and
\begin{gather*}
\sum_nE\Abs{E\bigl\{Q_{n+1}^{(i,j)}\mid\mathcal{F}_n\bigr\}-Q_n^{(i,j)}}=\sum_n\frac{E\bigl\{\abs{Q_n^{(i,j)}}\bigr\}}{(n+1)^2}
\\\le\frac{E\bigl\{Q_2^{(i,i)}\bigr\}+E\bigl\{Q_2^{(j,j)}\bigr\}}{2}\,\sum_n\frac{1}{(n+1)^2}<\infty.
\end{gather*}
To sum up, for fixed $i$ and $j$, the sequence $(Q_n^{(i,j)}:n\ge 2)$ is a quasi-martingale such that $\sup_n\,E\bigl\{\abs{Q_n^{(i,j)}}\bigr\}<\infty$, and this implies $Q_n^{(i,j)}\overset{a.s.}\longrightarrow Q^{(i,j)}$ for some random variable $Q^{(i,j)}$. Hence, $Q_n\overset{a.s.}\longrightarrow Q$ where $Q$ is the matrix $Q=\bigl(Q^{(i,j)}:1\le i,\,j\le p\bigr)$.

\medskip

\noindent\textbf{Step 2: $M_n$ is a martingale and it is bounded in $L_2$.}

\medskip

\noindent Just note that
\begin{gather*}
E\bigl(M_{n+1}-M_n\mid\mathcal{F}_n\bigr)=E\left\{\frac{Q_n^{1/2}Z_{n+1}}{n+1}\mid\mathcal{F}_n\right\}
\\=\frac{Q_n^{1/2}E(Z_{n+1}\mid\mathcal{F}_n)}{n+1}=\frac{Q_n^{1/2}E(Z_{n+1})}{n+1}=0.
\end{gather*}
Hence, $M_n$ is a martingale. In addition,
\begin{gather*}
E(M_n^tM_n)=E\left\{\left(\sum_{i=1}^n\frac{Q_{i-1}^{1/2}\,Z_i}{i}\right)^t\left(\sum_{j=1}^n\frac{Q_{j-1}^{1/2}\,Z_j}{j}\right)\right\}
=\sum_{i=1}^n\frac{E(Z_i^t\,Q_{i-1}\,Z_i)}{i^2}\\=\sum_{i=1}^n\,\,\sum_{r,k=1}^p\frac{E(Z_i^{(r)}\,Z_i^{(k)})\,E(Q_{i-1}^{(r,k)})}{i^2}
=\sum_{i=1}^n\sum_{k=1}^p\frac{E(Q_{i-1}^{(k,k)})}{i^2}
\\\le\sum_{i=1}^\infty\sum_{k=1}^p\frac{E(Q_{i-1}^{(k,k)})}{i^2}\le p + \frac{p}{4} + \sum_{k=1}^pE(Q_2^{(k,k)})\,\sum_{i=3}^\infty\frac{1}{i^2}
\end{gather*}
where the last inequality is because $(Q_n^{(k,k)}:n\ge 2)$ is a supermartingale for fixed $k=1,\ldots,p$. Therefore, $M_n$ is a martingale such that $\sup_nE(M_n^tM_n)<\infty$, and this implies that $M_n\overset{a.s.}\longrightarrow M$ for some $p$-dimensional random vector $M$.

\medskip

\noindent\textbf{Step 3: $Q\in\mathcal{M} $ a.s.}

\medskip

\noindent Since $Q\overset{a.s.}=\lim_nQ_n$, the matrix $Q$ is symmetric and non-negative definite a.s. Hence, to prove $Q\in\mathcal{M} $ a.s., it suffices to show that det$(Q)>0$ a.s.

\medskip

\noindent Exploiting \eqref{bh87f4r} and Lemma \ref{ji87c4}, one obtains
\begin{gather*}
\text{det}(Q_{n+1})=\text{det}\left(\frac{n}{n+1}\,Q_n^{1/2}\,\left(I+\frac{1}{(n+1)}\,Z_{n+1}Z_{n+1}^t\right)\,Q_n^{1/2}\right)
\\=\text{det}(Q_n)\,\left(1-\frac{1}{n+1}\right)^p\,\left(1+\frac{1}{n+1}\,Z_{n+1}^tZ_{n+1}\right).
\end{gather*}
Letting $V_n=Z_n^tZ_n-p$, the above equation can be written as
\begin{gather*}
\text{det}(Q_{n+1})=\text{det}(Q_n)\,\sum_{j=0}^p\left(
                                                   \begin{array}{c}
                                                     p \\
                                                     j \\
                                                   \end{array}
                                                 \right)
\,\left(\frac{-1}{n+1}\right)^j\,\left(1+\frac{V_{n+1}+p}{n+1}\right)
\\=\text{det}(Q_n)\,\left\{1+\frac{V_{n+1}}{n+1}-p\,\frac{V_{n+1}+p}{(n+1)^2}+R_{n+1}\right\}
\end{gather*}
where
$$R_{n+1}=\sum_{j=2}^p\left(
                                                   \begin{array}{c}
                                                     p \\
                                                     j \\
                                                   \end{array}
                                                 \right)
\,\left(\frac{-1}{n+1}\right)^j\,\left(1+\frac{V_{n+1}+p}{n+1}\right).$$
Therefore,
$$\text{det}(Q_n)=\text{det}(Q_2)\,\prod_{j=3}^n\,\frac{\text{det}(Q_j)}{\text{det}(Q_{j-1})}=\text{det}(Q_2)\,\prod_{j=3}^n\left\{1+\frac{V_j}{j}-p\,\frac{V_j+p}{j^2}+R_j\right\}.$$

\medskip

\noindent Next, define
$$U_j=1+\frac{V_j}{j}-p\,\frac{V_j+p}{j^2}+R_j.$$
If $\sum_{j=3}^n\log U_j\overset{a.s.}\longrightarrow T$ as $n\rightarrow\infty$, for some real random variable $T$, then
$$\text{det}(Q)=\lim_n\text{det}(Q_n)=\text{det}(Q_2)\,\lim_n\,\exp\Bigl(\,\sum_{j=3}^n\log U_j\Bigr)=\text{det}(Q_2)\,e^T>0\quad\quad\text{a.s.}$$
Hence, it suffices to show that the series $\sum_j\log U_j$ converges a.s. In turn, this is true since each series involved in the definition of $U_j$ converges a.s. Precisely, $\sum_j\frac{V_j}{j}$ converges a.s. because of Lemma \ref{7yu8}. The series $\sum_j\frac{V_j+p}{j^2}$ converges a.s. since $V_j+p=Z_j^tZ_j\ge 0$ and
$$E\left\{\sum_j\frac{V_j+p}{j^2}\right\}=\sum_j\,\frac{E(Z_j^tZ_j)}{j^2}=\sum_j\frac{p}{j^2}<\infty.$$
Finally, $\sum_jR_j$ converges a.s. by exactly the same argument used for $\sum_j\frac{V_j+p}{j^2}$. Therefore, $\sum_j\log U_j$ converges a.s., and this concludes the proof.

\end{proof}

\bigskip

\begin{proof}[\textbf{Proof of Theorem \ref{de4v6}}]
Let $\norm{\cdot}_E$ be the Euclidean norm on $\mathbb{R}^p$ and $\norm{\cdot}_F$ the Frobenius norm on the matrices of order $p\times p$. The latter is defined as
$$\norm{B}_F=\sqrt{\sum_{i,j}\bigl(B^{(i,j)}\bigr)^2}=\sqrt{\text{trace}\,(B^tB)}\quad\quad\text{for any }p\times p\text{ matrix }B$$
and has the nice property that $\norm{B_1B_2}_F\le\norm{B_1}_F\,\norm{B_2}_F$.

\medskip

\noindent The total variation distance between two $p$-dimensional normal laws, say $\mathcal{N}(a,\Sigma_1)$ and $\mathcal{N}(b,\Sigma_2)$ with $a,\,b\in\mathbb{R}^p$ and $\Sigma_1,\,\Sigma_2\in\mathcal{M}$, can be estimated as
$$\norm{\mathcal{N}(a,\Sigma_1)-\mathcal{N}(b,\Sigma_2)}\le\norm{\mathcal{N}(a,\Sigma_1)-\mathcal{N}(b,\Sigma_1)}+\frac{3}{2}\,\norm{\Sigma_1^{-1/2}\Sigma_2\Sigma_1^{-1/2}-I}_F;$$
see e.g. \cite[Theorem 2]{KELB} and \cite[Formula (2)]{DMR}. In addition,
$$\norm{\Sigma_1^{-1/2}\Sigma_2\Sigma_1^{-1/2}-I}_F=\norm{\Sigma_1^{-1/2}(\Sigma_2-\Sigma_1)\Sigma_1^{-1/2}}_F\le\norm{\Sigma_1^{-1/2}}_F^2\,\norm{\Sigma_2-\Sigma_1}_F$$
and
$$\norm{\mathcal{N}(a,\Sigma_1)-\mathcal{N}(b,\Sigma_1)}\le\norm{b-a}_E\,\norm{\Sigma_1^{-1/2}}_F;$$
see \cite[Page 5]{DMR}. Collecting all these facts together,
$$\norm{\mathcal{N}(a,\Sigma_1)-\mathcal{N}(b,\Sigma_2)}\le C(\Sigma_1)\,\Bigl\{\norm{b-a}_E+\norm{\Sigma_2-\Sigma_1}_F\Bigr\}$$
where $C(\Sigma_1)$ is a constant which depends on $\Sigma_1$ only.

\medskip

\noindent Applying the previous inequality to our framework, one obtains
$$\norm{\alpha_n-\alpha}=\norm{\mathcal{N}(M_n,Q_n)-\mathcal{N}(M,Q)}\le C(Q)\,\Bigl\{\norm{M_n-M}_E+\norm{Q_n-Q}_F\Bigr\}.$$
Hence, it suffices to show that
$$b:=\sup_n\sqrt{n}\,E\Bigl\{\norm{M_n-M}_E+\norm{Q_n-Q}_F\Bigr\}<\infty.$$
In fact, if $b<\infty$ and $d_n$ is any sequence of constants such that $\frac{d_n}{\sqrt{n}}\rightarrow 0$, then
$$d_n\,E\Bigl\{\norm{M_n-M}_E+\norm{Q_n-Q}_F\Bigr\}\le\frac{b\,d_n}{\sqrt{n}}\longrightarrow 0.$$
In particular, $d_n\,\Bigl\{\norm{M_n-M}_E+\norm{Q_n-Q}_F\Bigr\}\overset{P}\longrightarrow 0$, which in turn implies
$$d_n\,\norm{\alpha_n-\alpha}\le d_n\,C(Q)\,\Bigl\{\norm{M_n-M}_E+\norm{Q_n-Q}_F\Bigr\}\overset{P}\longrightarrow 0.$$

\medskip

\noindent To prove $b<\infty$, we first note that
$$M-M_n=\sum_{i=n}^\infty(M_{i+1}-M_i)=\sum_{i=n}^\infty\frac{Q_i^{1/2}Z_{i+1}}{i+1}.$$
Therefore, for each $n\ge 2$,
\begin{gather*}
E\Bigl\{\norm{M_n-M}_E\Bigr\}^2\le E\Bigl\{\norm{M_n-M}_E^2\Bigr\}=E\Bigl\{(M-M_n)^t(M-M_n)\Bigr\}
\\=\sum_{i=n}^\infty\frac{E(Z_{i+1}^t\,Q_i\,Z_{i+1})}{(i+1)^2}=\sum_{i=n}^\infty\,\,\sum_{r,\,k=1}^p\frac{E(Z_{i+1}^{(r)}\,Z_{i+1}^{(k)})\,E(Q_i^{(r,k)})}{(i+1)^2}
\\=\sum_{i=n}^\infty\,\,\sum_{k=1}^p\frac{E(Q_i^{(k,k)})}{(i+1)^2}\le\sum_{k=1}^pE(Q_2^{(k,k)})\,\sum_{i=n}^\infty\,\,\frac{1}{(i+1)^2}
\end{gather*}
where the last inequality is because $\{Q_i^{(k,k)}:i\ge 2\}$ is a supermartingale for fixed $k$. On noting that $\sum_{i=n}^\infty\,\,\frac{1}{(i+1)^2}\le\frac{1}{n}$, one obtains
$$\sqrt{n}\,E\Bigl\{\norm{M_n-M}_E\Bigr\}\le\sqrt{\sum_{k=1}^pE(Q_2^{(k,k)})}.$$
Similarly, recalling that $Q_{n+1}=\frac{n}{n+1}\,Q_n^{1/2}\bigl(I+\frac{Z_{n+1}Z_{n+1}^t}{n+1}\bigr)Q_n^{1/2}$, one obtains
\begin{gather*}
Q-Q_n=\sum_{i=n}^\infty(Q_{i+1}-Q_i)
\\=\sum_{i=n}^\infty\left\{\frac{Q_i^{1/2}(Z_{i+1}Z_{i+1}^t-I)Q_i^{1/2}}{i+1}-\frac{Q_i^{1/2}Z_{i+1}Z_{i+1}^tQ_i^{1/2}}{(i+1)^2}\right\}.
\end{gather*}
Define
$$q=\sup_iE\bigl\{\norm{Q_i^{1/2}}_F^4\bigr\}\quad\text{and}\quad K_i=\frac{Q_i^{1/2}(Z_{i+1}Z_{i+1}^t-I)Q_i^{1/2}}{i+1}.$$
Since $E\bigl\{\norm{Q_i^{1/2}}_F^2\bigr\}\le\sqrt{q}$ for all $i$,
\begin{gather*}
\sum_{i=n}^\infty\frac{E\bigl\{\norm{Q_i^{1/2}Z_{i+1}Z_{i+1}^tQ_i^{1/2}}_F\bigr\}}{(i+1)^2}\le\sum_{i=n}^\infty\frac{E\bigl\{\norm{Z_{i+1}Z_{i+1}^t}_F\bigr\}\,E\bigl\{\norm{Q_i^{1/2}}_F^2\bigr\}}{(i+1)^2}
\\\le E\bigl\{\norm{Z_1Z_1^t}_F\bigr\}\,\sum_{i=n}^\infty\frac{\sqrt{q}}{(i+1)^2}\le\frac{\sqrt{q}}{n}\,E\bigl\{\norm{Z_1Z_1^t}_F\bigr\}.
\end{gather*}
Since $K_i$ is symmetric and $E(K_iK_j)=0$ for $i\ne j$,
\begin{gather*}
E\Bigl\{\norm{\sum_{i=n}^\infty K_i}_F^2\Bigr\}= E\left\{\text{trace}\left(\Bigl(\sum_{i=n}^\infty K_i\Bigr)\,\Bigl(\sum_{j=n}^\infty K_j\Bigr)\right)\right\}
\\=\text{trace}\Bigl(\,\sum_{i,j=n}^\infty E(K_iK_j)\Bigr)=\text{trace}\Bigl(\,\sum_{i=n}^\infty E(K_iK_i)\Bigr)
\\=\sum_{i=n}^\infty E\left\{\frac{\norm{Q_i^{1/2}(Z_{i+1}Z_{i+1}^t-I)Q_i^{1/2}}_F^2}{(i+1)^2}\right\}
\\\le E\bigl\{\norm{Z_1Z_1^t-I}_F^2\bigr\}\,\sum_{i=n}^\infty\frac{E\bigl\{\norm{Q_i^{1/2}}_F^4\bigr\}}{(i+1)^2}\le E\bigl\{\norm{Z_1Z_1^t-I}_F^2\bigr\}\,\frac{q}{n}.
\end{gather*}
Hence,
\begin{gather*}
\sqrt{n}\,E\bigl\{\norm{Q_n-Q}_F\bigr\}\le \sqrt{n}\,\sum_{i=n}^\infty\frac{E\bigl\{\norm{Q_i^{1/2}Z_{i+1}Z_{i+1}^tQ_i^{1/2}}_F\bigr\}}{(i+1)^2}+\sqrt{n}\,E\bigl\{\norm{\sum_{i=n}^\infty K_i}_F\bigr\}
\\\le\frac{\sqrt{q}}{\sqrt{n}}\,E\bigl\{\norm{Z_1Z_1^t}_F\bigr\}+\sqrt{q\,E\bigl\{\norm{Z_1Z_1^t-I}_F^2\bigr\}}.
\end{gather*}

\medskip

\noindent Hence, $b<\infty$ provided $q<\infty$. We finally prove $q<\infty$.

\medskip

\noindent Fix $i\ge 2$, $1\le k\le p$, and define $W_i=\bigl(Q_i^{(k,k)}\bigr)^2$. As noted in the proof of Theorem \ref{q2098nfer}, the conditional distribution of $L_j=Q_j^{1/2}Z_{j+1}$ given $\mathcal{F}_j$ is $\mathcal{N}(0,Q_j)$. Hence,
$$E\bigl\{\bigl(L_j^{(k)})^2\mid\mathcal{F}_j\bigr\}=Q_j^{(k,k)}\quad\text{and}\quad E\bigl\{\bigl(L_j^{(k)})^4\mid\mathcal{F}_j\bigr\}=3\,\bigl(Q_j^{(k,k)}\bigr)^2=3\,W_j.$$
Using \eqref{bh87f4r}, it follows that
$$E(W_{j+1}\mid \mathcal{F}_j)=W_j\,\left\{\frac{j^2}{(j+1)^2}+\frac{3\,j^2}{(j+1)^4}+\frac{2\,j^2}{(j+1)^3}\right\}\le W_j.$$
Therefore,
\begin{gather*}
E(W_i)=E\left\{W_2\,\prod_{j=2}^{i-1}\frac{W_{j+1}}{W_j}\right\}=E(W_2)\,\prod_{j=2}^{i-1}E\left\{\frac{W_{j+1}}{W_j}\right\}
\\=E(W_2)\,\prod_{j=2}^{i-1}E\left\{\frac{E(W_{j+1}\mid\mathcal{F}_j)}{W_j}\right\}\le E(W_2).
\end{gather*}
To sum up, $E\bigl\{\bigl(Q_i^{(k,k)}\bigr)^2\bigr\}\le E\bigl\{\bigl(Q_2^{(k,k)}\bigr)^2\bigr\}$ for all $i\ge 2$ and $1\le k\le p$. Hence,
\begin{gather*}
E\bigl\{\norm{Q_i^{1/2}}_F^4\bigr\}=E\bigl\{\text{trace}(Q_i)^2\bigr\}=\sum_{r,k=1}^pE\bigl\{Q_i^{(r,r)}Q_i^{(k,k)}\bigr\}
\\\le\sum_{r,k=1}^p\sqrt{E\bigl\{\bigl(Q_i^{(r,r)}\bigr)^2\bigr\}\,E\bigl\{\bigl(Q_i^{(k,k)}\bigr)^2\bigr\}}
\\\le p^2\max_{1\le k\le p}E\bigl\{\bigl(Q_i^{(k,k)}\bigr)^2\bigr\}\le p^2\max_{1\le k\le p}E\bigl\{\bigl(Q_2^{(k,k)}\bigr)^2\bigr\}.
\end{gather*}
This proves that $q<\infty$ and concludes the proof of the theorem.
\end{proof}

\bigskip

\noindent Let us turn to the results of Section \ref{x6h8k1q}. In such results, the predictives of $X$ are the copula-based predictive distributions $\beta_n$. Hence, as noted in Section \ref{x6h8k1q}, $X$ is conditionally identically distributed.

\medskip

\begin{proof}[\textbf{Proof of Theorem \ref{g7j09i}}]
As in Example \ref{bh78u}, define
$$D_n=\int_{-\infty}^\infty\,\Abs{f_{n+1}\bigl(x\mid X(n+1)\bigr)-f_n\bigl(x\mid X(n)\bigr)}\,dx.$$
For $n<m$,
\begin{gather*}
\int_{-\infty}^\infty\,\Abs{f_n\bigl(x\mid X(n)\bigr)-f_m\bigl(x\mid X(m)\bigr)}\,dx
\\=\int_{-\infty}^\infty\,\Abs{\,\sum_{i=n}^{m-1}\Bigl\{f_{i+1}\bigl(x\mid X(i+1)\bigr)-f_i\bigl(x\mid X(i)\bigr)\Bigr\}}\,dx\le\sum_{i=n}^{m-1}D_i.
\end{gather*}
If $\sum_nD_n<\infty$ a.s., the sequence $f_n\bigl(\cdot\mid X(n)\bigr)$ is Cauchy in the $L^1$-norm a.s., so that it converges in the $L^1$-norm a.s. Hence, it suffices to show that $\sum_nD_n<\infty$ a.s. To this end, recall that $D_n$ is distributed as $r_n\int_0^1 \abs{c_n(u,\,V)-1}\,du$ where $V$ is any random variable with uniform distribution on $(0,1)$; see Example \ref{bh78u}. It follows that
\begin{gather*}
E(D_n)=r_n\,E\left\{\int_0^1 \abs{c_n(u,\,V)-1}\,du\right\}=r_n\,\int_0^1 E\bigl\{\abs{c_n(u,\,V)-1}\bigr\}\,du
\\=r_n\,\int_0^1\int_0^1\abs{c_n(u,v)-1}\,dv\,du\le 2\,r_n.
\end{gather*}
Therefore, $\sum_nE(D_n)\le 2\,\sum_nr_n<\infty$ which in turn implies $\sum_nD_n<\infty$ a.s.
\end{proof}

\medskip

\begin{proof}[\textbf{Proof of Theorem \ref{r58uh6}}]
Since $X$ is conditionally identically distributed, by \cite[Theorem 4]{AOP13}, it suffices to show that
\begin{gather}\label{v6yh9m}
\sup_nE\left\{\,\int_{-\infty}^\infty\,f_n\bigl(x\mid X(n)\bigr)^2\,dx\right\}<\infty.
\end{gather}
To prove \eqref{v6yh9m}, we note that, since $X$ is conditionally identically distributed,
$$E\Bigl\{f_{k+1}\bigl(x\mid X(k+1)\bigr)\mid X(k)\Bigr\}=f_k\bigl(x\mid X(k)\bigr)$$
for almost every $x\in\mathbb{R}$ (with respect to Lebesgue measure). This implies
$$E\left\{c_k\Bigl(F_k\bigl(x\mid X(k)\bigr),\,F_k\bigl(X_{k+1}\mid X(k)\bigr)\Bigr)\mid X(k)\right\}=1$$
for almost every $x\in\mathbb{R}$. Since $c_k(u,v)\le b$ for all $u,\,v\in [0,1]$, it follows that
\begin{gather*}
E\left\{\left(\frac{f_{k+1}\bigl(x\mid X(k+1)\bigr)}{f_k\bigl(x\mid X(k)\bigr)}\right)^2\mid X(k)\right\}\le (1-r_k)^2+r_k^2\,b^2\,+2\,(1-r_k)\,r_k
=1+r_k^2(b^2-1).
\end{gather*}
Finally, the previous inequality implies
\begin{gather*}
E\left\{f_n\bigl(x\mid X(n)\bigr)^2\right\}=E\left\{\left( f_0(x)\,\prod_{k=0}^{n-1}\frac{f_{k+1}\bigl(x\mid X(k+1)\bigr)}{f_k\bigl(x\mid X(k)\bigr)}\right)^2\right\}
\\\le f_0(x)^2\,\prod_{k=0}^{n-1}\bigl\{1+r_k^2(b^2-1)\bigr\}\le f_0(x)^2\,\exp\bigl\{(b^2-1)\,\sum_{k=0}^{n-1}r_k^2\bigr\}.
\end{gather*}
Thus, to check condition \eqref{v6yh9m}, it suffices noting that
\begin{gather*}
E\left\{\,\int_{-\infty}^\infty\,f_n\bigl(x\mid X(n)\bigr)^2\,dx\right\}=\int_{-\infty}^\infty E\left\{f_n\bigl(x\mid X(n)\bigr)^2\right\}\,dx
\\\le\exp\bigl\{(b^2-1)\,\sum_{k=0}^\infty r_k^2\bigr\}\,\,\int_{-\infty}^\infty f_0(x)^2dx.
\end{gather*}
\end{proof}

\end{document}